\documentclass[a4paper,12pt,oneside]{article}

\usepackage[latin1]{inputenc}
\usepackage{eucal}
\usepackage{exscale}
\usepackage{latexsym}
\usepackage[T1]{fontenc}
\usepackage{amsmath}
\usepackage{graphicx}
\usepackage{amsfonts,makeidx,color,amsmath, amssymb,graphics,mathrsfs,eurosym,amsthm}
\usepackage{epsfig}

\newtheorem{theorem}{Theorem}

\newcommand{\RR}{{\mathbb{R}}}
\newcommand{\ZZ}{{\mathbb{Z}}}

\newcommand{\PP}{{\mathbb{P}}}

\title{\bf Near-best $C^2$ quartic spline quasi-interpolants on type-6 tetrahedral partitions of bounded domains}
\author{C. Dagnino, P. Lamberti and S. Remogna\thanks{Department of Mathematics, University of Torino, via C. Alberto 10, 10123 Torino, Italy ({\tt catterina.dagnino@unito.it, paola.lamberti@unito.it, sara.remogna@unito.it})}}

\date{}
\begin{document}

\maketitle

\begin{abstract}
In this paper, we present new quasi-interpolating spline schemes defined on 3D bounded domains, based on trivariate $C^2$ quartic box splines on type-6 tetrahedral partitions and with approximation order four. Such methods can be used for the reconstruction of gridded volume data. More precisely, we propose near-best quasi-interpolants, i.e. with coefficient functionals obtained by imposing the exactness of the quasi-interpolants on the space of polynomials of total degree three and minimizing an upper bound for their infinity norm. In case of bounded domains the main problem consists in the construction of the coefficient functionals associated with boundary generators (i.e. generators with supports not completely inside the domain), so that the functionals involve data points inside or on the boundary of the domain. 

We give norm and error estimates and we present some numerical tests, illustrating the approximation properties of the proposed quasi-interpolants, and comparisons with other known spline methods. Some applications with real world volume data are also provided.
\end{abstract}\vskip 10pt


\section{Introduction} \label{intro}

We consider the problem of constructing appropriate non-discrete models on bounded domains from given discrete volume data. The development of such trivariate models is relevant because it is the theoretical basis for many applications, such as scientific visualization, computer graphics, medical imaging, numerical simulation, seismic applications, etc. Indeed, volume data sets typically represent some kind of density acquired by special devices that often require structured input data, so that the samples are arranged on a regular three-dimensional grid. It is important to underline that we need to construct models defined on bounded domains of $\RR^3$, because we have at our disposal a finite number of volumetric data. 

In classical approaches the underlying mathematical models use local trivariate tensor-product polynomial splines, defined on bounded domains as linear combinations of univariate B-spline products. If we require a certain smoothness along the coordinate axes, such splines can be of high coordinate degree, that can create unwanted oscillations and often require (approximate) derivative data at certain prescribed points. These reasons raise the natural problem of constructing alternative smooth spline models, that use only data values on the volumetric grid and simultaneously approximate smooth functions as well as their derivatives. Moreover, in order to avoid unwanted oscillations, it is desirable that polynomial sections have total degree, instead of the coordinate degree, that is typical of tensor product schemes.

A possible 3D spline model, beyond the classical tensor product schemes, is represented by blending sums of univariate and bivariate $C^1$ quadratic spline quasi-interpolants (see e.g. \cite{r1,s1,s2}), defined on a bounded domain, partitioned into vertical prisms with triangular horizontal sections. In this case the approximation order is three, as for triquadratic splines, but the degree of the piecewise polynomials is only four, instead of six.

In order to have smooth splines of lower total degree, other methods based on trivariate splines defined on type-6 tetrahedral partitions of bounded domains have been proposed.  In literature \cite{nrsz,SZ}, quasi-interpolating schemes, based on either quadratic or cubic $C^1$ splines on type-6 tetrahedral partitions, are described and the piecewise polynomials are directly determined by setting their Bernstein-B\'ezier coefficients (BB-coefficients) to appropriate combinations of the data values. The above splines yield approximation order two.

In this paper we propose and analyse new quasi-interpolating spline sche\-m\-es, on type-6 tetrahedral partitions of bounded domains, with higher $C^2$ smoothness and approximation order four. They are expressed as linear combination of scaled translates of the trivariate $C^2$ quartic box spline $B$ \cite{p1}, defined on such a partition, and local functionals. These functionals are defined as linear combinations of data values. Such quasi-interpolating operators are exact on the space $\PP_3$ of trivariate polynomials of total degree at most three and, therefore, yield approximation order four.

We recall that a fundamental property of such quasi-interpolants (abbr. QIs) is that they do not need the solution of huge systems of linear equations, with respect to the approximation by interpolating operators. This is very effective in the 3D space.

Among the different techniques that can be used to define QIs with such properties, here we decide to construct operators of near-best type (i.e. with coefficient functionals obtained by minimizing an upper bound for their infinity norm), motivated by the good results obtained in \cite{ABIS,BISS,BISS2,BISS1,BISS3,I,racm,r} for the univariate and bivariate settings. The main problem in their construction consists in finding the coefficient functionals associated with boundary generators, i.e. scaled translate box splines with supports not completely inside the domain.

We remark that trivariate splines with $C^2$ smoothness can be used for constructing curvature continuous surfaces by tracing their zero sets.

Here is an outline of the paper. In Section \ref{spl} we define the space $S_4^2(\Omega, \mathcal{T}_{{\mathbf m}})$ of $C^2$ quartic splines on a type-6 tetrahedral partition $\mathcal{T}_{{\mathbf m}}$ of a parallelepiped $\Omega$ and present its properties. In Section \ref{QI}, we construct near-best spline QIs and we provide norm and error estimates for smooth functions. Finally, in Section \ref{num}, in order to illustrate the approximation properties of the proposed quasi-interpolants, some numerical tests are presented and compared with those obtained by some other known trivariate spline quasi-interpolants. We also provide some examples with real world volume data.

\section{The spline space $S_4^2(\Omega, \mathcal{T}_{{\mathbf m}})$}\label{spl}

A box spline is specified by a set of direction vectors that determines the shape of its support and also its smoothness properties. We consider the set of seven direction vectors of $\ZZ^3$, spanning $\RR^3$ \cite{p1}
$$
X=\{e_1, e_2, e_3, e_4, e_5, e_6, e_7\},
$$
where
$$
\begin{array}{c}
e_1=(1,0,0), \quad e_2=(0 ,1 ,0 ),\quad  e_3=(0 ,0 ,1), \quad e_4=(1 ,1 ,1), \\
e_5=(-1 ,1 ,1), \quad e_6=(1 ,-1 ,1), \quad e_7=(-1 ,-1 ,1).
\end{array}
$$

Such direction vectors subdivide $\RR^3$ into the so-called type-6 tetrahedral partition, obtained from a given cube partition of the space (see Fig. \ref{part}$(a)$), by subdividing each cube into 24 tetrahedra (see Fig. \ref{part}$(b)$). Since cuts with six planes are required to subdivide the cubes in this way, the resulting tetrahedral partition is called type-6 tetrahedral partition.

\begin{figure}[ht]
\begin{minipage}{60mm}
\centering\includegraphics[height=3.8cm]{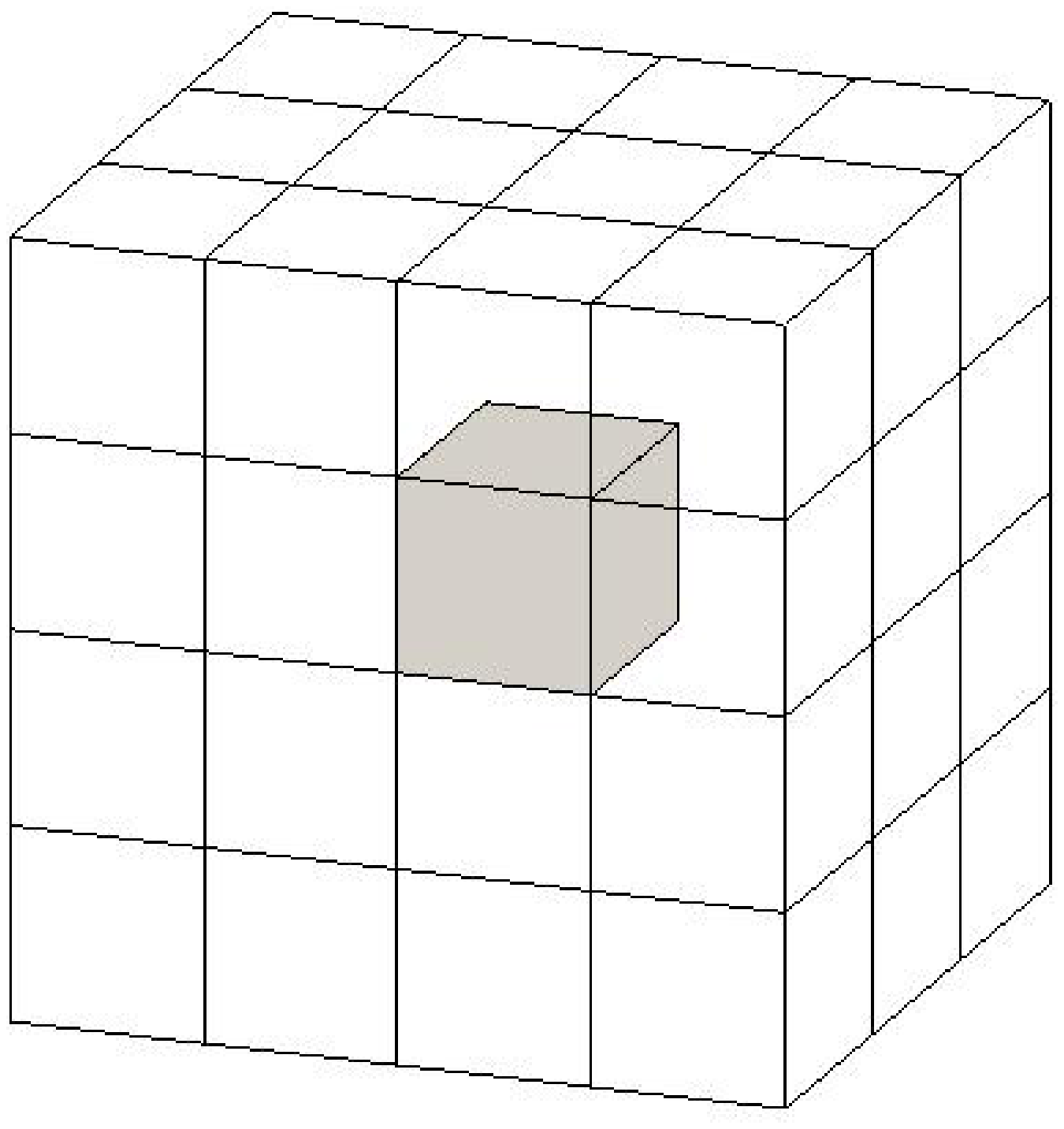}
\centerline{$(a)$}
\end{minipage}
\hfil
\begin{minipage}{60mm}
\centering\includegraphics[height=3.8cm]{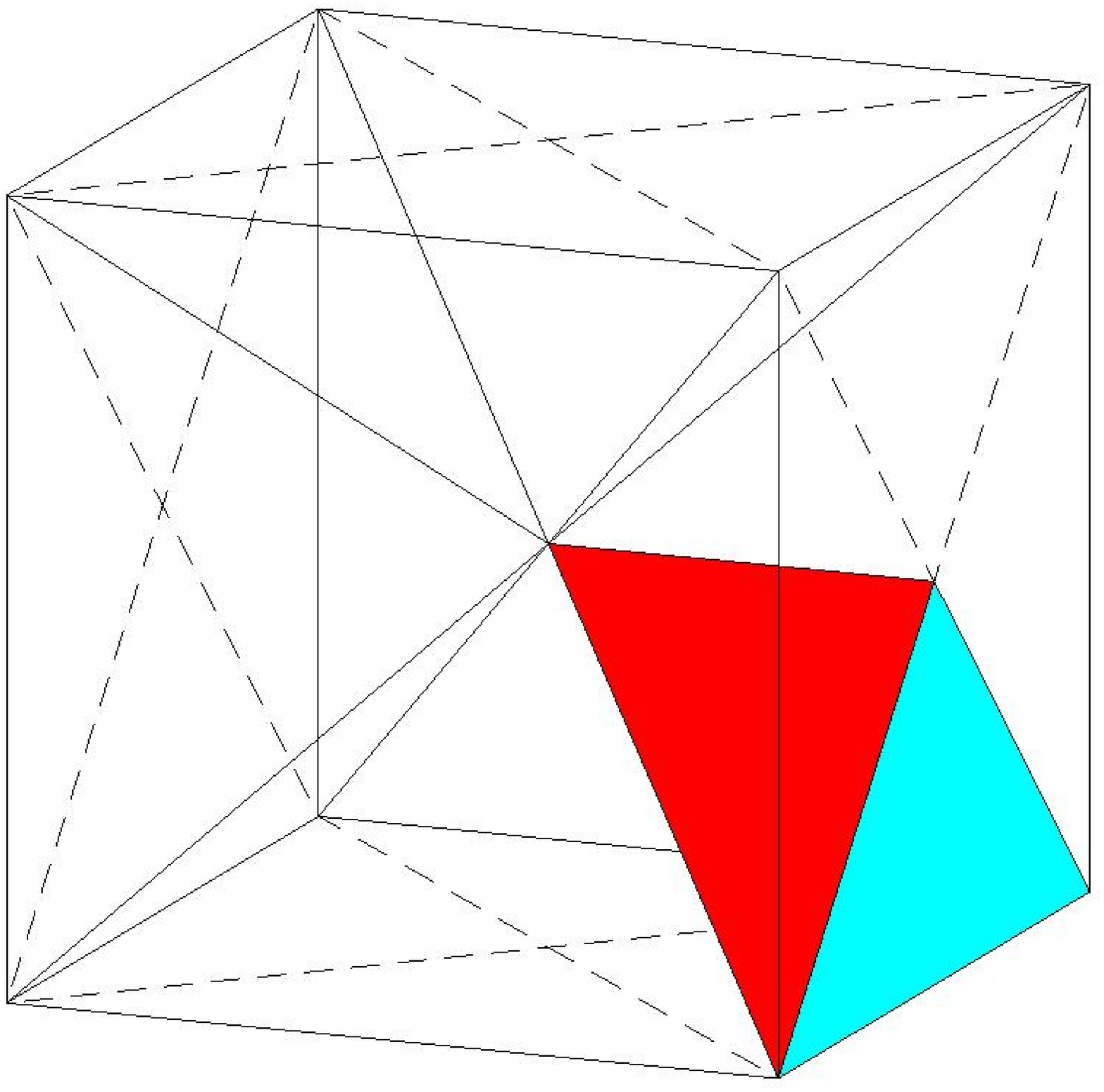}
\centerline{$(b)$}
\end{minipage}
\caption{$(a)$ Cube partition. $(b)$ Uniform type-6 tetrahedral partition, obtained by subdividing each cube into 24 tetrahedra}
	\label{part}
\end{figure}

As the set $X$ has seven elements in $\RR^3$, the box spline $B(\cdot)=B(\cdot|X)$ is of degree four (see e.g. \cite[Chap. 11]{bhs}, \cite[Chap. 1]{dhr} and \cite[Chap. 15,16,17]{LS}). Its smoothness depends on the determination of the number $d$, such that $d+1$ is the minimal number of directions to be removed from $X$ to obtain a reduced set, that does not span $\RR^3$. Then, one deduces that the smoothness class is $C^{d-1}$. In our case $d=3$, thus the polynomial pieces defined over each tetrahedron are of degree four and they are joined with $C^2$ smoothness.

The support of $B$ is the truncated rhombic dodecahedron centered at the point $(\frac12,\frac12,\frac52)$ and contained in the cube $[-2,3]\times[-2,3]\times[0,5]$ (see Fig. \ref{supp}).

\begin{figure}
\centering\includegraphics[width=6cm]{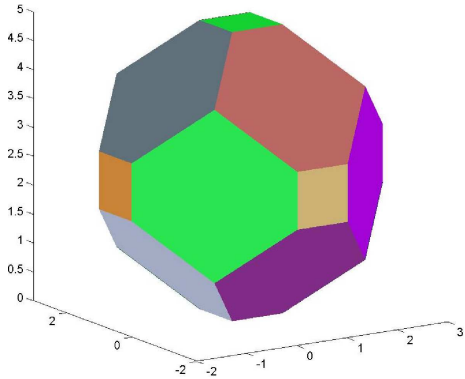}
\caption{The support of the seven directional box spline $B$}
	\label{supp}
\end{figure}

Now, let $m_1$, $m_2$ and $m_3$ be positive integers and let $\Omega=[0,m_1h] \times [0,m_2h] \times [0,m_3h]$, $h>0$, be a parallelepiped subdivided into $m_1m_2m_3$ equal cubes and endowed with the type-6 tetrahedral partition $\mathcal{T}_{{\mathbf m}}$, ${\mathbf m}=(m_1,m_2,m_3)$ (see Fig. \ref{part}). Let 
\begin{equation}\label{aa}
\mathcal{A}= \left\{ \begin{array}{lll}
							  & -1\le i \le m_1+2, &\\
\alpha=(i,j,k), & -1\le j \le m_2+2, & \alpha \notin \mathcal{A}^\prime\\
							  & -1\le k \le m_3+2; &
									    \end{array}\right\},
\end{equation}
with $\mathcal{A}^\prime$ the set of indices defined by
$$
\begin{array}{lll}
\mathcal{A}^\prime= & \{ (i,-1,-1), \; (i,m_2+2,-1),&                               \\ 
               & (i,-1,m_3+2), \; (i,m_2+2,m_3+2), & {\rm for } \; -1\le i\le m_1+2, \\
               & (-1,j,-1), \; (m_1+2,j,-1),   &                               \\ 
               & (-1,j,m_3+2), \; (m_1+2,j,m_3+2), & {\rm for } \; 0\le j\le m_2+1,  \\
               & (-1,-1,k), \; (m_1+2,-1,k),   &                               \\
               & (-1,m_2+2,k), \; (m_1+2,m_2+2,k), & {\rm for } \; 0\le k\le m_3+1 \}.
\end{array}
$$

Since $B$ has centre at the point $\left(\frac12,\frac12,\frac52\right)$, we define the scaled translates of $B$, $\{ B_{\alpha}, \alpha \in \mathcal{A}\}$, in the following way:
$$
B_{\alpha}(x,y,z)=B_{i,j,k}(x,y,z)=B\left(\frac{x}{h}-i+1,\frac{y}{h}-j+1,\frac{z}{h}-k+3\right),
$$
whose supports are centered at the points 
\begin{equation}\label{center}
C_\alpha=C_{i,j,k}=\left(\left(i-\frac12\right)h,\left(j-\frac12\right)h,\left(k-\frac12\right)h\right).
\end{equation}
The set $\mathcal{A}$, defined in (\ref{aa}), is the index set of scaled translates of $B$ whose supports overlap with $\Omega$.

Then, we define the space generated by the B-splines $\{ B_\alpha, \alpha \in \mathcal{A}\}$
$$
S_4^2(\Omega, \mathcal{T}_{{\mathbf m}})=\left\{s=\sum_{\alpha \in  \mathcal{A}} a_\alpha B_\alpha, \; \; a_\alpha \in \RR \right\},
$$
that is a subspace of the whole space of all trivariate $C^2$ quartic splines defined on $\mathcal{T}_{{\mathbf m}}$. Moreover, it is well-known that $\PP_3 \subset S_4^2(\Omega, \mathcal{T}_{{\mathbf m}})$ and $\PP_4 \not\subset S_4^2(\Omega, \mathcal{T}_{{\mathbf m}})$.

We also recall \cite[Chap. 3]{dhr} that the approximation power of $S_4^2(\Omega, \mathcal{T}_{{\mathbf m}})$ is the largest integer $r$ for which
$$
\mbox{dist}(f,S_4^2(\Omega, \mathcal{T}_{{\mathbf m}}))= O(h^r)
$$
for all sufficiently smooth $f$, with the distance measured in the $L_p(\Omega)$-norm ($1 \leq p \leq \infty$). In our case we get $r=4$.

\section{Quasi-interpolation operators in $S_4^2(\Omega, \mathcal{T}_{{\mathbf m}})$}\label{QI}

In the space $S_4^2(\Omega, \mathcal{T}_{{\mathbf m}})$ we consider quasi-interpolants of the form
\begin{equation}\label{oQI}
	Qf=\sum_{\alpha\in \mathcal{A}} \lambda_\alpha (f)  B_\alpha,
\end{equation}
where $f \in C(\Omega)$ and $\lambda_\alpha (f)$ are linear combinations of values of $f$ at specific points inside $\Omega$ or on its boundary $\partial \Omega$.

The data point set, used in the construction of the coefficient functionals $\lambda_\alpha (f)$, is 
\begin{equation}\label{Dset}
	\mathcal{D} = \{ M_\alpha=M_{i,j,k}=(s_i,t_j,u_k), \, (i,j,k) \in \mathcal{A}^M\},
\end{equation}

with $\mathcal{A}^M=\{(i,j,k),\, 0\le i\le m_1+1,\, 0\le j\le m_2+1,\, 0\le k\le m_3+1\}$ and
$$
\begin{array}{llll}
s_0=0, & s_{i}=(i-\frac12)h, & 1\le i\le m_1, &  s_{m_1+1}=m_1h; \\ 
t_0=0, & t_{j}=(j-\frac12)h, & 1\le j\le m_2, &  t_{m_2+1}=m_2h; \\ 
u_0=0, & u_{k}=(k-\frac12)h, & 1\le k\le m_3, &  u_{m_3+1}=m_3h.
\end{array}
$$
The values of the function $f$ at the above points are denoted by $f_\alpha=f(M_{\alpha})$. In Fig. \ref{ma} we report some data points $M_\alpha \in \mathcal{D}$, depending on the position of cubes in $\Omega$. We remark that they are all inside $\Omega$ or on $\partial \Omega$. In particular, the points lying on $\partial \Omega$ (see e.g. Fig. \ref{ma} $(b)$, $(c)$, $(d)$) can be thought of as points outside $\Omega$ and projected on $\partial \Omega$.

\begin{figure}[ht]
\begin{minipage}{25mm}
\centering\includegraphics[width=3cm]{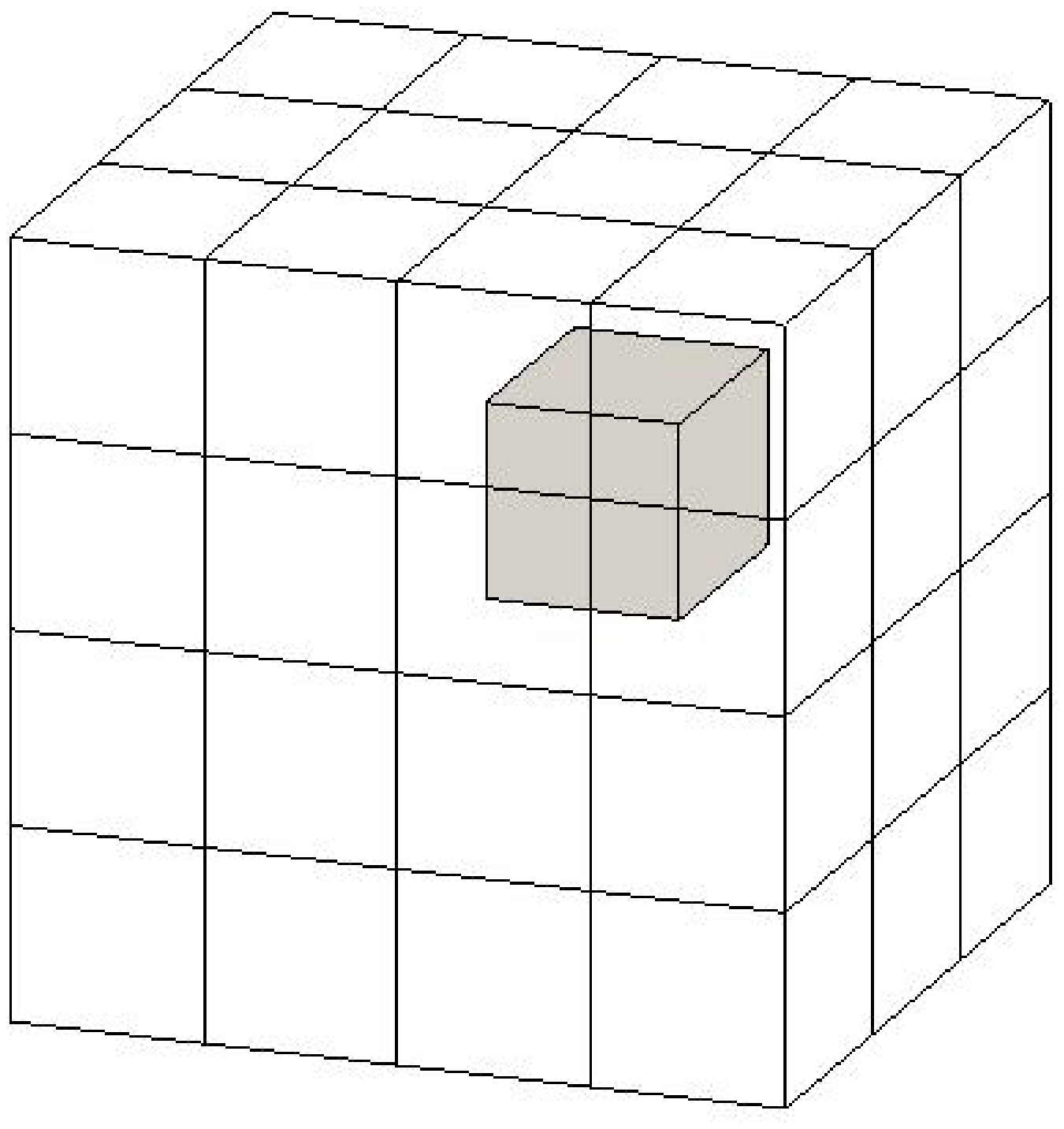}
\end{minipage}
\begin{minipage}{25mm}
\centering\includegraphics[width=1.5cm]{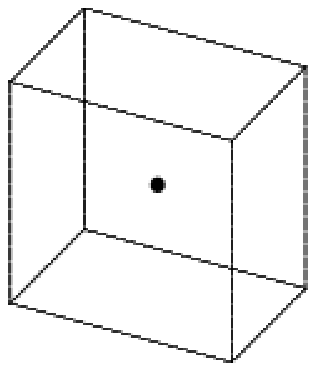}
\end{minipage}
\begin{minipage}{25mm}
\centering\includegraphics[width=3cm]{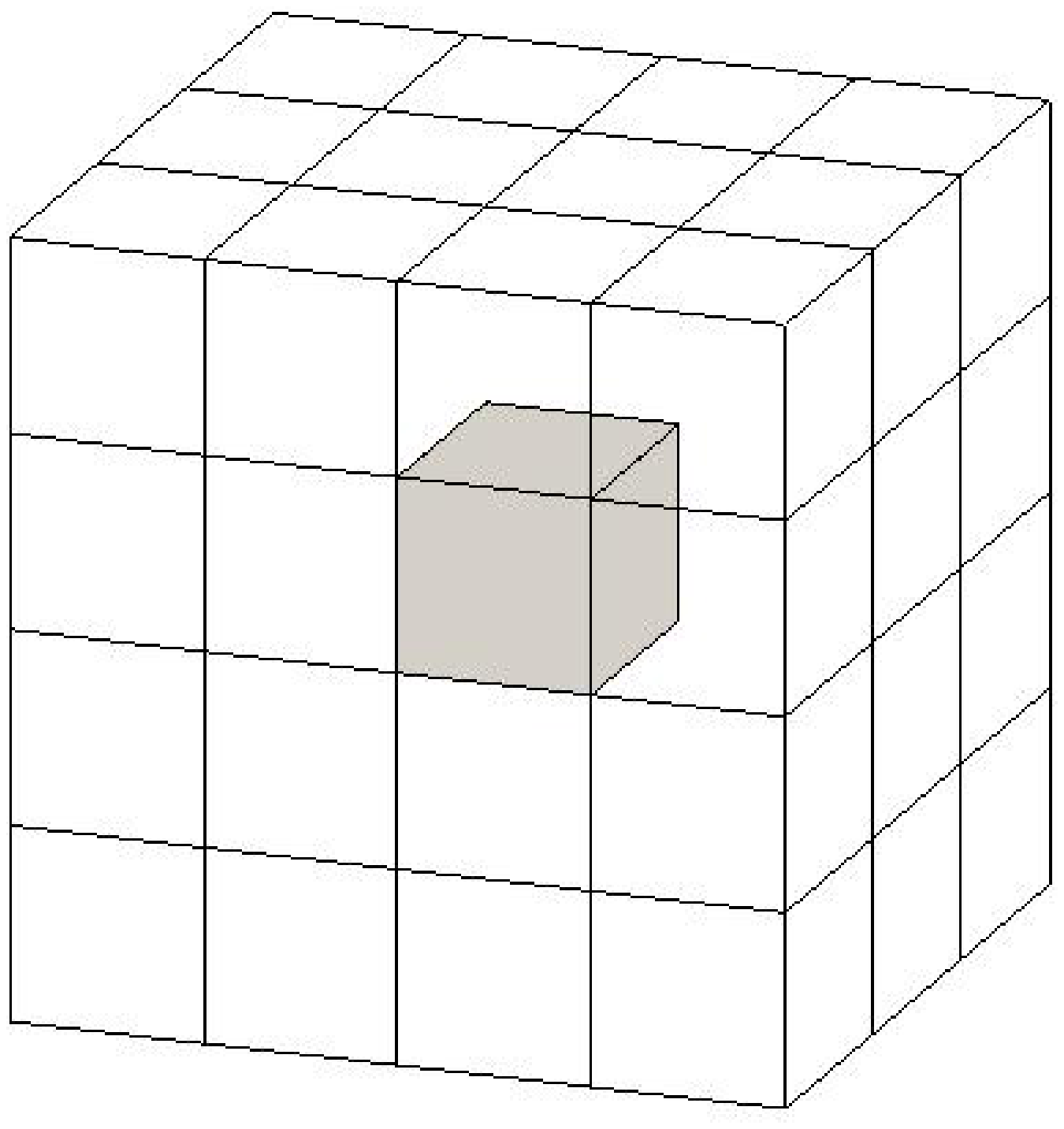}
\end{minipage}
\begin{minipage}{25mm}
\centering\includegraphics[width=1.5cm]{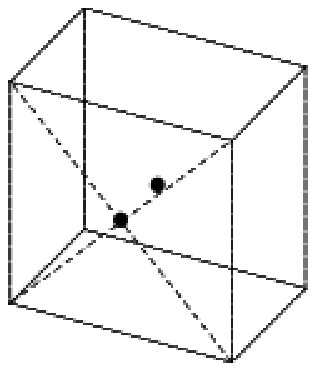}
\end{minipage}
\hfil
\begin{minipage}{100mm}
\centerline{\hskip2.5cm $(a)$ \hskip5.5cm $(b)$}
\end{minipage}
\hfil
\begin{minipage}{25mm}
\centering\includegraphics[width=3cm]{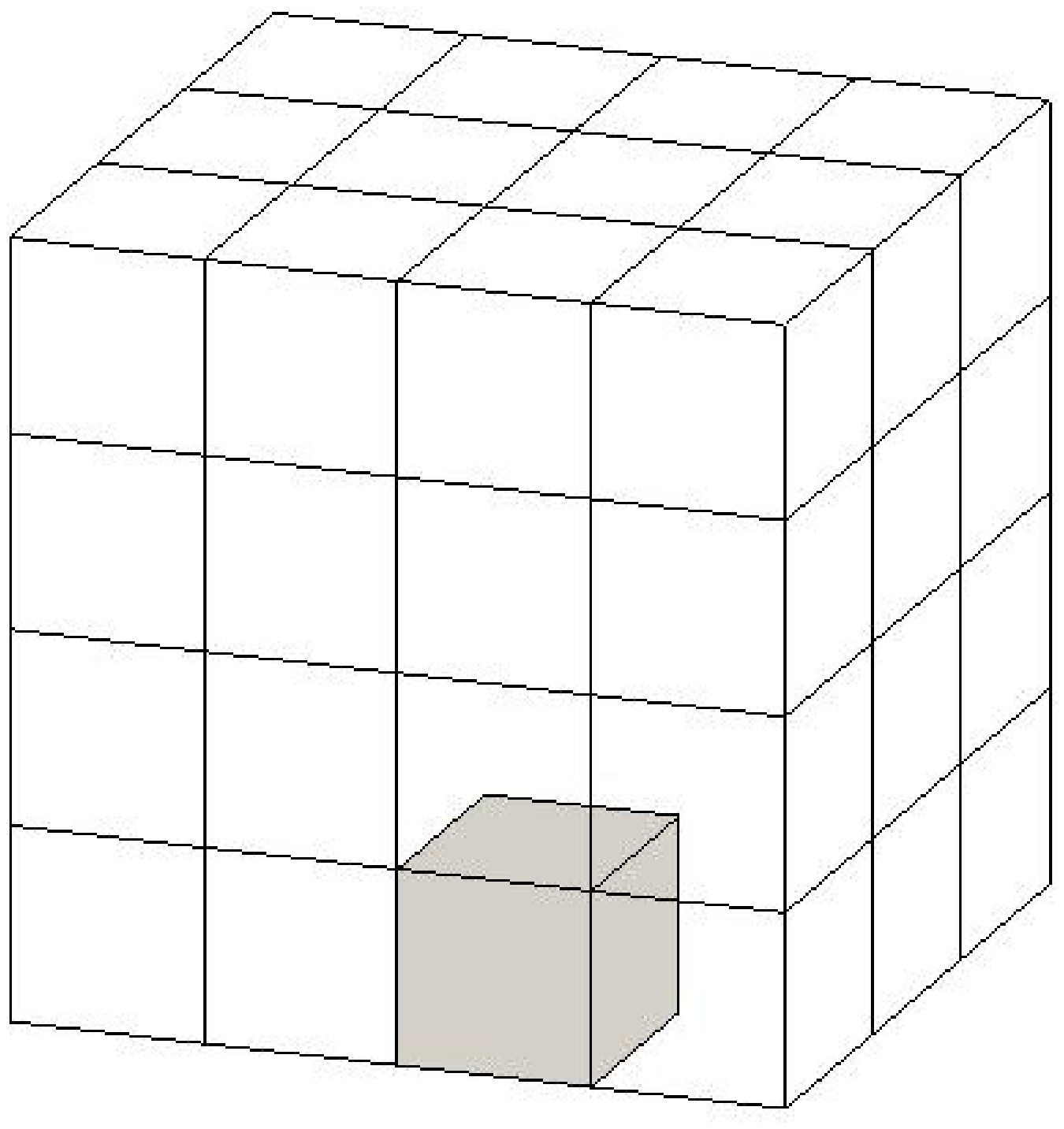}
\end{minipage}
\begin{minipage}{25mm}
\centering\includegraphics[width=1.5cm]{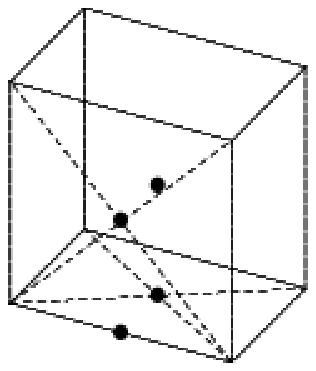}
\end{minipage}
\begin{minipage}{25mm}
\centering\includegraphics[width=3cm]{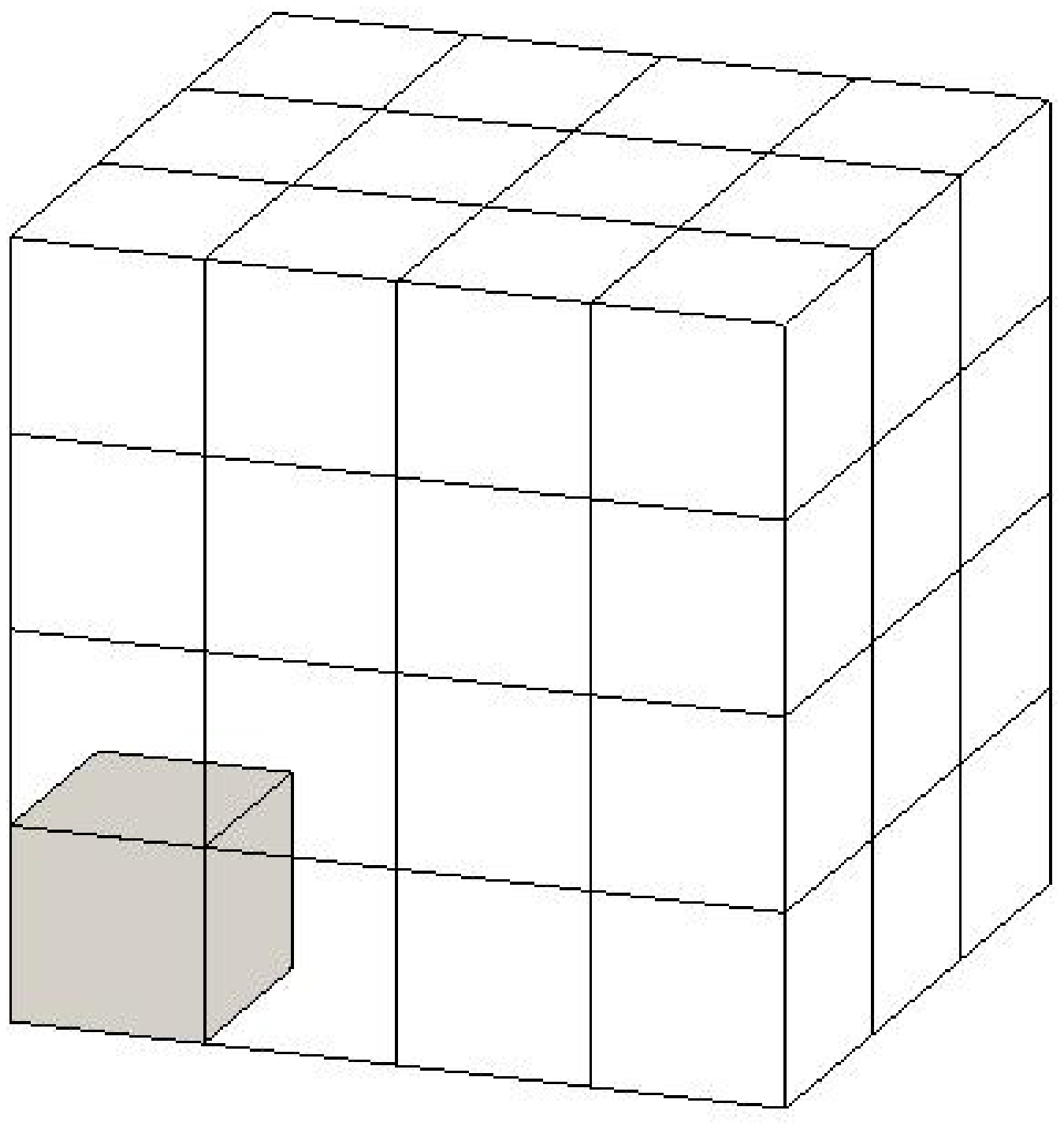}
\end{minipage}
\begin{minipage}{25mm}
\centering\includegraphics[width=1.5cm]{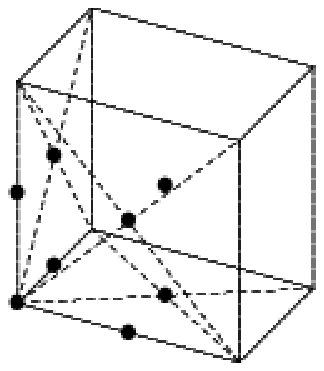}
\end{minipage}
\hfil
\begin{minipage}{100mm}
\centerline{\hskip2.5cm $(c)$ \hskip5.5cm $(d)$}
\end{minipage}
\hfil
\caption{Data points $M_\alpha$'s for $(a)$ inner, $(b)$ face, $(c)$ edge, $(d)$ vertex cubes}
	\label{ma}
\end{figure}

We also notice that recently, in \cite{Rbit}, QIs based on the same box spline $B$ have been proposed, but they are defined on the whole space $\RR^3$. Therefore, if only a finite number of volume data is available, such schemes can reconstruct only a portion of the volumetric object.

The coefficient functionals $\lambda_\alpha (f)$ in (\ref{oQI}) have the following expression
\begin{equation}\label{ll}
\lambda_\alpha(f) =\sum_{\beta \in F_{\alpha}}\sigma_{\alpha}(\beta)f_\beta, 
\end{equation}
where the finite set of points $\left\{ M_\beta,\ \beta \in F_{\alpha},\ F_{\alpha} \subset \mathcal{A}^M\right\}$ lies in some neighbourhood of the support of $B_\alpha \cap \Omega$ and $\{\sigma_{\alpha}(\beta)\}$ are real numbers, obtained so that $Qf\equiv f$ for all $f$ in $\PP_3$.

In the definition of the functionals, we consider a number of data points greater than the number of conditions we are imposing, in order to obtain a system of equations with free parameters, that we choose by minimizing an upper bound for the infinity norm of the operator $Q$.

\subsection{A general presentation of the optimization process}
 
From (\ref{ll}), it is clear that, for $\|f\|_\infty \leq 1$ and $\alpha \in \mathcal{A}$, $\left|\lambda_{\alpha}(f)\right|\leq \|\sigma_{\alpha}\|_1$, where $\sigma_{\alpha}$ is the vector with components $\sigma_{\alpha}(\beta)$. Therefore, since the scaled translates of $B$ form a partition of unity \cite{bhs}, we deduce immediately 
\begin{equation} \label{nn}
\|Q \|_\infty\leq\max_{\alpha \in \mathcal{A}}\left|\lambda_{\alpha}(f)\right|\leq \max_{\alpha \in \mathcal{A}}\|\sigma_{\alpha}\|_1.
\end{equation}
Now, we can try to find a solution of the minimization problem
\begin{equation}\label{sys}
\min\left\{\|\sigma_{\alpha}\|_1: \sigma_\alpha \in \RR^{{\rm card}(F_{\alpha})}, V_{\alpha}\sigma_{\alpha}=b_{\alpha}\right\},
\end{equation}
where $ V_{\alpha}\sigma_{\alpha}=b_{\alpha}$ is the linear system expressing that $Q$ is exact on $\PP_3$. In our case, in order to obtain such a system, we require that, for $f\in \PP_3$, each coefficient functional coincides with the corresponding one of the differential quasi-interpolating operator \cite{Rbit}, constructed by the general method proposed in \cite{dm} and exact on $\PP_3$
$$
\widehat{Q}f = \displaystyle\sum_{\alpha \in \ZZ^3} \left(I-\frac{5}{24}h^2\Delta +\frac{3}{128}h^4\Delta^2 \right)f(C_\alpha) B_{\alpha}.
$$

This is an $l_1$-minimization problem and there are many well-known techniques for approximating the solution, not unique in general (see e.g. \cite[Chap. 6]{Wa}). Since the minimization problem is equivalent to a linear programming one, here we use the simplex method.

\subsection{Choice of data points used in coefficient functionals}
Now, we explain the logic behind our approach, proposing a general method for the construction of coefficient functionals (\ref{ll}).

Fixed the scaled translate box spline $B_\alpha=B_{i,j,k}$, for any $n \geq 1$ we define $\Lambda^n_{\alpha}=\Lambda^n_{i,j,k}$ the octahedron centered at the point $C_{i,j,k}$, given in (\ref{center}), with vertices $C_{i\pm n,j,k}$, $C_{i,j\pm n,k}$, $C_{i,j,k\pm n}$. We denote by $\bar{\Lambda}^n_{\alpha}=\bar{\Lambda}^n_{i,j,k}$ the set of points $\Lambda^n_{i,j,k}\cap \left[\left(\ZZ-\frac12\right)h\right]^3$ (see Fig. \ref{los} for $n=2$ and $3$). 

The choice of octahedron shapes can be considered a 3D generalization of the rhombic one used in \cite{BISS,I} for the construction of 2D near-best QIs.

\begin{figure}[ht]
\begin{minipage}{60mm}
\centering\includegraphics[width=4cm]{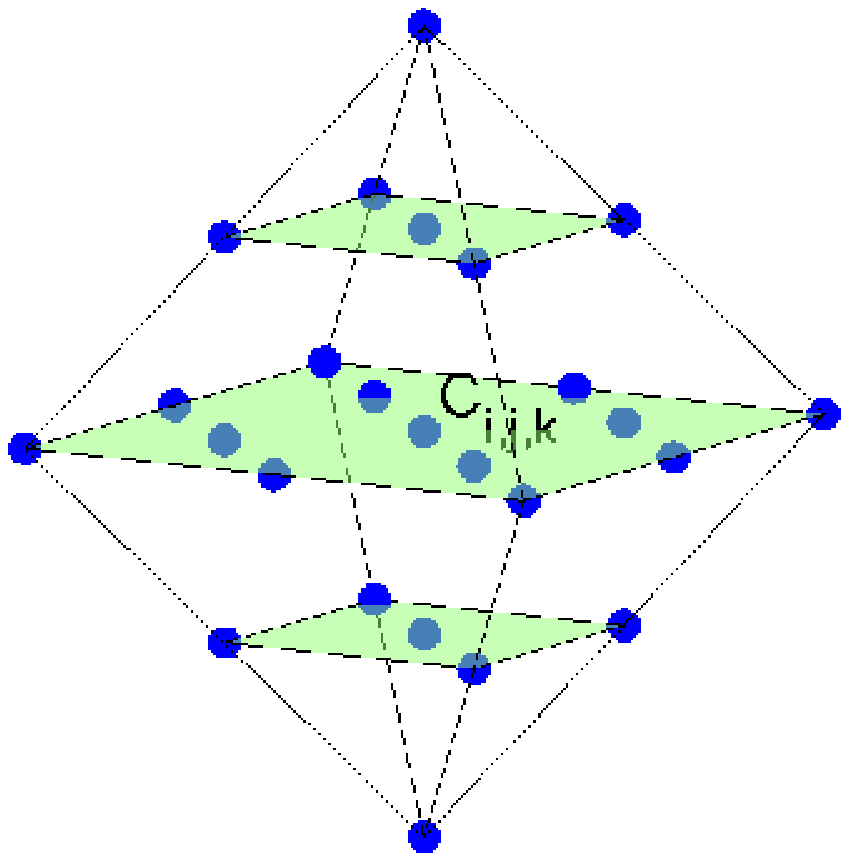}
\vskip2cm
\centerline{$(a)$}
\end{minipage}
\hfil
\begin{minipage}{60mm}
\centering\includegraphics[width=6cm]{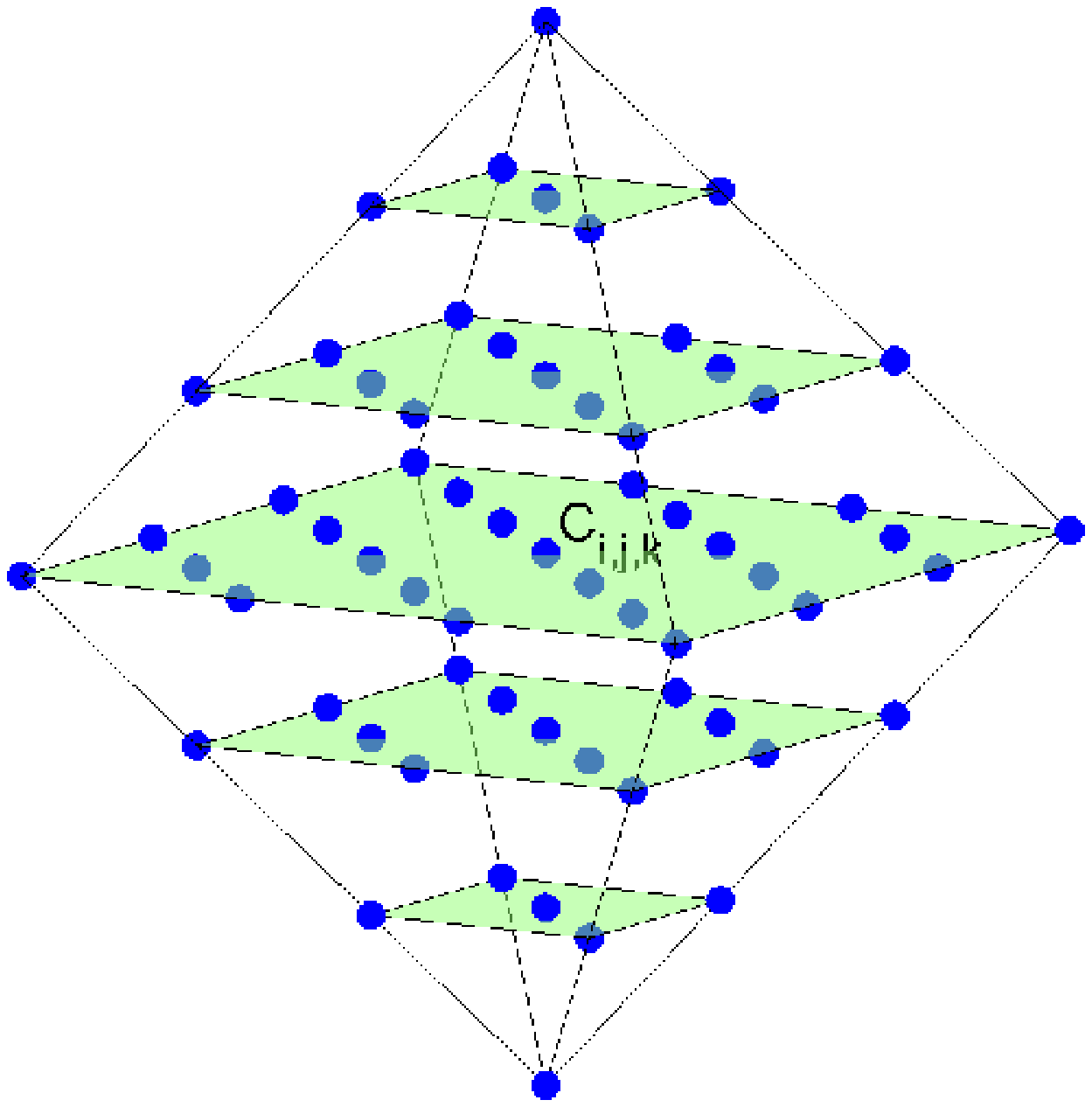}
\centerline{$(b)$}
\end{minipage}
\caption{The octahedron $\Lambda^n_{i,j,k}$ with the corresponding points $\bar{\Lambda}^n_{i,j,k}$, for $(a)$ $n=2$ and $(b)$ $n=3$} 
	\label{los}
\end{figure}

If, for a fixed $n$, we used the points in $\bar{\Lambda}^n_{\alpha}$, as data points in the definition (\ref{ll}) of the coefficient functional $\lambda_\alpha (f)$, some of them could lie outside $\Omega$. Therefore, in order to have all data points inside or on the boundary of $\Omega$, we project them on $\partial \Omega$, obtaining a new set that we still call $\bar{\Lambda}^n_{\alpha}$. We remark that $\bar{\Lambda}^n_{\alpha} \subset  \mathcal{D}$, defined in (\ref{Dset}).

Then, for the construction of $\lambda_\alpha (f)$, we impose the exactness of $Q$ on $\PP_3$, getting twenty conditions (or less than twenty in case of symmetry). Thus, we choose a starting value for $n$, we construct a scheme for the coefficient functional based on the data point set $\bar{\Lambda}^n_{\alpha}$ and containing at least twenty unknown parameters (or less than twenty in case of symmetry). We solve the resulting linear system and, in case of free parameters, we minimize the norm $\|\sigma_{\alpha}\|_1$. 

In order to reduce such a norm, we increase the value of $n$ and we repeat the process above explained.

\subsection{Detailed presentation of a particular case}

We describe this method in detail in the case $\alpha=(0,0,-1)$.

We impose $\lambda_{0,0,-1}(f)\equiv (I-\frac{5}{24}h^2\Delta+\frac{3}{128}h^4\Delta^2)f(C_{0,0,-1})$, for all monomials of $\PP_3$, i.e. for $f\equiv 1$, $x$, $y$, $z$, $x^2$, $y^2$, $z^2$, $xy$, $xz$, $yz$, $x^3$, $y^3$, $z^3$, $x^2y$, $xy^2$, $x^2z$, $xz^2$, $y^2z$, $yz^2$, $xyz$. Due to the symmetry of the support of $B_{0,0,-1}$ with respect to the plane $y=x$, we have only 13 conditions to impose. Indeed the monomials $y$, $y^2$, $yz$, $xy^2$, $y^2z$, $yz^2$ and $y^3$ can be excluded. 

If we consider the data points in $\bar{\Lambda}^n_{0,0,-1}$,  $n=1$, 2, 3 (with the points lying outside $\Omega$ projected on $\partial \Omega$), the corresponding systems have not solution. With $n=4$, taking into account the symmetry, we obtain a coefficient functional of the form 
$$
\begin{array}{ll}
\lambda_{0,0,-1}(f)= & \sigma_{0,0,-1}(0,0,0)f_{0,0,0} + \sigma_{0,0,-1}(1,0,0)(f_{1,0,0}+f_{0,1,0}) + \\
                     & \sigma_{0,0,-1}(2,0,0)(f_{2,0,0}+f_{0,2,0}) + \sigma_{0,0,-1}(3,0,0)(f_{3,0,0}+f_{0,3,0}) +\\
                     & \sigma_{0,0,-1}(4,0,0)(f_{4,0,0}+f_{0,4,0}) + \sigma_{0,0,-1}(1,1,0)f_{1,1,0} + \\
                     & \sigma_{0,0,-1}(2,1,0)(f_{2,1,0}+f_{1,2,0}) + \sigma_{0,0,-1}(3,1,0)(f_{3,1,0}+f_{1,3,0}) + \\
                     & \sigma_{0,0,-1}(2,2,0)f_{2,2,0} + \sigma_{0,0,-1}(0,0,1)f_{0,0,1} + \\
                     & \sigma_{0,0,-1}(1,0,1)(f_{1,0,1}+f_{0,1,1}) + \sigma_{0,0,-1}(2,0,1)(f_{2,0,1}+f_{0,2,1}) + \\
                     & \sigma_{0,0,-1}(1,1,1)f_{1,1,1} + \sigma_{0,0,-1}(0,0,2)f_{0,0,2} + \\
                     & \sigma_{0,0,-1}(1,0,2)(f_{1,0,2}+f_{0,1,2}) + \sigma_{0,0,-1}(0,0,3)f_{0,0,3}.
\end{array}
$$

\begin{figure}[ht]
\centering\includegraphics[width=9cm]{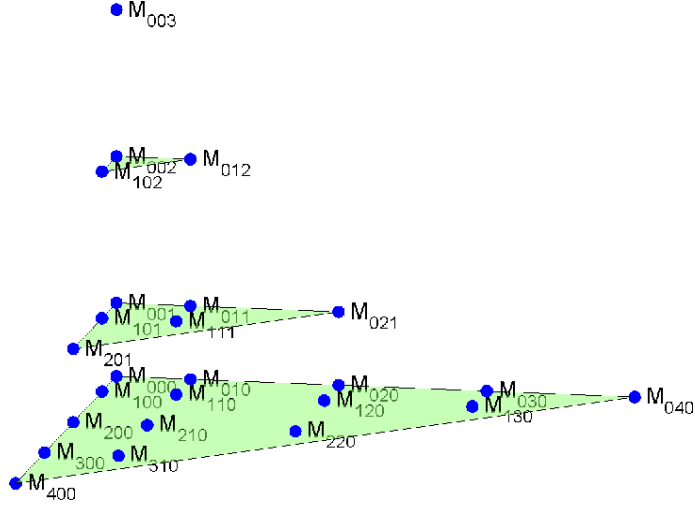}
\caption{The data point set $\bar{\Lambda}^4_{0,0,-1}$}
    \label{l4}
\end{figure}

In Fig. \ref{l4}, the data point set $\bar{\Lambda}^4_{0,0,-1}$ is provided. By this choice of $\lambda_{0,0,-1}(f)$, we obtain a linear system whose solution has three free parameters, that, from (\ref{sys}), we determine by minimizing $\| \sigma_{0,0,-1}\|_1$, getting the values
$$
\begin{array}{l}
\sigma_{0,0,-1}(0,0,0)=26956/945, \; \sigma_{0,0,-1}(1,0,0)=-331/36,\\
\sigma_{0,0,-1}(2,0,0)=-181/216, \; \sigma_{0,0,-1}(3,0,0)=0, \; \sigma_{0,0,-1}(4,0,0)=11/504,\\
\sigma_{0,0,-1}(1,1,0)=0, \; \sigma_{0,0,-1}(2,1,0)=25/18, \; \sigma_{0,0,-1}(3,1,0)=0 , \\
\sigma_{0,0,-1}(2,2,0)=-13/27, \; \sigma_{0,0,-1}(0,0,1)=-827/24, \; \sigma_{0,0,-1}(1,0,1)=35/3, \\
\sigma_{0,0,-1}(2,0,1)=-1/6 , \; \sigma_{0,0,-1}(1,1,1)=-3, \; \sigma_{0,0,-1}(0,0,2)=337/36,\\
\sigma_{0,0,-1}(1,0,2)=-31/18, \; \sigma_{0,0,-1}(0,0,3)=-151/120,
\end{array}
$$
and $\| \sigma_{0,0,-1}\|_1 \approx 127.1$. In order to construct a functional with a smaller upper bound for its norm, we consider increasing values for $n$ and the corresponding data points in $\bar{\Lambda}^n_{0,0,-1}$. Then, we solve the system and we minimize $\| \sigma_{0,0,-1}\|_1$. In the first row of Table \ref{tabnorm} the values of $\| \sigma_{0,0,-1}\|_1$ for $n \leq 11$ are shown. We notice that $\| \sigma_{0,0,-1}\|_1$ reduces when more data points are added, i.e. for increasing values for $n$, but this reduction slows down as $n$ increases. 

\bigskip

Using the same logic, we get the other coefficient functionals: we start with an initial $\bar{\Lambda}^n_{\alpha}$ such that the system has solution and, in order to reduce $\| \sigma_\alpha\|_1$, we consider increasing values of $n$. For each value we solve the corresponding system and we minimize the norm $\| \sigma_\alpha\|_1$.

\subsection{Discussion of the choice of $n$}
Since a large  number of QIs can be constructed, we have to choose a specific $n$ for each coefficient functional in (\ref{oQI}). For this choice we have taken into account the following remarks: since the norm reduction slows down as $n$ increases, we have to strike a balance between $n$ and the value $\| \sigma_\alpha\|_1$, in order to keep the data points in a neighbourhood of the support of $B_\alpha \cap \Omega$. 

Therefore, following this approach and taking into account the values of $\| \sigma_\alpha\|_1$ and $n$, given in Table \ref{tabnorm}, we define a particular QI, whose coefficient functionals are given in Tables \ref{tab-3}, \ref{tab-4}, \ref{tab1} and \ref{tab12}. We remark that, for this particular choice of the coefficient functionals, we need $m_1, m_2, m_3 \geq 11$ to be able to construct them.

We have only defined the coefficient functionals associated with the $B_\alpha$'s whose centers of the supports are close to either the origin $O=(0,0,0)$ or the edge $OT$ of $\Omega$, with $T=(m_1h,0,0)$, because the other ones can be obtained from them by symmetry.

\begin{table}[ht]
	\caption{Values of $\| \sigma_\alpha\|_1$, for increasing values of $n$}
\begin{tabular}{|c||c|c|c|c|c|c|c|c|c|c|c|}
\hline   &  \multicolumn{11}{c|}{$n$}  \\
\hline  $(i,j,k)$ & 1 & 2 & 3 & 4 & 5 & 6 & 7 & 8 & 9 & 10 & 11 \\
\hline 
$(0,0,-1)$ & -- & -- & -- & 127.1 & 55.27 & 29.28 & 20.13 & 15.37 & 12.37 & 10.25 & 8.774\\
\hline 
$(1,0,-1)$ & -- & -- & -- & 68.69 & 34.37 & 19.71 & 14.07 & 11.01 & 9.099 & 7.684 & 6.672\\
\hline 
$(2,0,-1)$ & -- & -- & -- & 68.69 & 34.37 & 19.71 & 14.07 & 11.01 & 9.099 & 7.684 &  6.672\\
\hline 
$(1,1,-1)$ & -- & -- & -- & 32.44 & 19.78 & 12.59 & 9.386 & 7.523 & 6.439 & 5.561 & 4.911 \\
\hline 
$(2,1,-1)$ & -- & -- & -- & 32.44 & 19.78 & 12.59 & 9.386 & 7.523 & 6.439 & 5.561 & 4.911  \\
\hline 
$(2,2,-1)$ & -- & -- & -- & 32.44 & 19.78 & 12.59 & 9.386 & 7.523 & 6.439 &  5.561 & 4.911 \\
\hline 
\hline 
$(0,0,0)$ & -- & -- & 30.09 & 17.35 & 10.31 & 7.740 & 6.486 & 5.639 & 4.945 & &\\
\hline 
$(1,0,0)$ & -- & -- & 11.04 & 7.649 & 5.492 & 4.463 & 3.928 & 3.570 & 3.237 & &\\
\hline 
$(2,0,0)$ & -- & -- & 9.945 & 7.649 & 5.435 & 4.443 & 3.928 & 3.569 & 3.237 & &\\
\hline 
$(3,0,0)$ & -- & -- & 9.945 & 7.649 & 5.435 & 4.443 & 3.928 & 3.569 & 3.237 & &\\
\hline 
$(1,1,0)$ & -- & -- & 5.508 & 3.787 & 3.077 & 2.621 & 2.318 & 2.143 & 2.009 & &\\
\hline 
$(2,1,0)$ & -- & -- & 5.108 & 3.518 & 2.912 & 2.502 & 2.251 & 2.113 & 1.983 & &\\
\hline 
$(3,1,0)$ & -- & -- & 5.048 & 3.469 & 2.880 & 2.477 & 2.247 & 2.109 & 1.981 & &\\
\hline
$(2,2,0)$ & -- & -- & 4.129 & 3.128 & 2.665 & 2.389 & 2.194 & 2.069 & 1.950 & &\\
\hline 
$(3,2,0)$ & -- & -- & 4.028 & 3.102 & 2.649 & 2.357 & 2.178 & 2.064 & 1.944 & &\\
\hline 
$(4,2,0)$ & -- & -- & 3.994 & 3.081 &  2.648 & 2.350 & 2.175 & 2.064 & 1.943 & &\\
\hline 
$(3,3,0)$ & -- & -- & 3.806 & 3.077 &  2.617 & 2.339 & 2.161 & 2.059 & 1.941 & &\\
\hline 
\hline 
$(1,1,1)$ & -- & 4.5   & 2.875 & 2.271 & 1.956 & 1.730 & 1.498 &  & & &\\
\hline 
$(2,1,1)$ & -- & 3.75  & 2.582 & 2.124 & 1.790 & 1.565 & 1.424 & & & &\\
\hline 
$(3,1,1)$ & -- & 3.542 & 2.536 & 2.114 & 1.771 & 1.526 & 1.380 & & & &\\
\hline 
$(2,2,1)$ & -- & 3.167 & 2.370 & 1.867 & 1.585 & 1.417 & 1.327 & & & & \\
\hline 
\hline 
$(2,2,2)$ & 3.5 & 2.25 & 1.732 & 1.494 & 1.384 & 1.297 & 1.244 & & & & \\
\hline 
\hline 
$(3,3,3)$ & 3.5 & 1.625 & 1.375 & 1.313 & 1.232 & 1.186 & 1.162 &  & & & \\
\hline 
\hline 
\end{tabular}
	\label{tabnorm}
\end{table}

\begin{table}[ht]
	\caption{Coefficient functionals $\lambda_{i,j,k}$ with the corresponding values of $n$ and $\|\sigma_{i,j,k}\|_1$ for $k=-1$}
\renewcommand{\arraystretch}{1.3}
\centering\begin{tabular}{|c|l|c|c|}
\hline 
$(i,j,k)$   &   \multicolumn{1}{|c|}{$\lambda_{i,j,k} (f)$} & $n$ & $\Vert \sigma_{i,j,k}\Vert_1$  \\
\hline \multicolumn{4}{|c|}{$k=-1$}\\
\hline 
\rule{0pt}{4ex} $(0,0,-1)$ & $\frac{5720029937968}{1777075925625}f_{0,0,0}-\frac{17625172171}{30540510000}(f_{3,0,0}+f_{0,3,0})+\frac{5091473}{125966750}f_{4,4,0}$ & 11 & 8.774\\
 & $-\frac{49957799237}{1496484990000}(f_{11,0,0}+f_{0,11,0})+\frac{42683993}{462735000}(f_{8,1,0}+f_{1,8,0})$      &  & \\
 & $-\frac{51197831}{3054051000} (f_{6,5,0}+f_{5,6,0})-\frac{323423}{157500}f_{0,0,3}+\frac{371}{1800}(f_{5,0,3}+f_{0,5,3})$ &  &   \\
 & $-\frac{3}{175}f_{3,3,4}-\frac{26}{165}(f_{4,0,6}+f_{0,4,6})+\frac{155}{312}(f_{1,0,7}+f_{0,1,7})$     &  & \\
 & $+\frac{557}{15000}f_{0,0,8}-\frac{6553}{26600}f_{0,0,10}$   &  &    \\
\hline 
\rule{0pt}{4ex} $(1,0,-1)$& $\frac{17446153}{20540520}f_{0,0,0}+\frac{7677660701}{3308104800}f_{1,0,0}+\frac{2896225}{6918912}f_{2,0,0}+\frac{772241}{5915669760}f_{10,0,0}$ & 9 &  9.099\\
				 & $-\frac{4139}{9072}f_{0,3,0}-\frac{109793}{453600}f_{2,3,0}+\frac{3743}{39312}f_{0,7,0}+\frac{16889}{157248}f_{1,7,0}$ & & \\
         & $+\frac{3041}{157248}f_{3,7,0}-\frac{473}{5712}f_{1,9,0}-\frac{815}{432}f_{0,0,2}-\frac{3997}{4320}f_{2,0,2}$ & & \\
         & $+\frac{1}{12}f_{1,3,3}+\frac{13}{100}f_{2,3,3}+\frac{13}{42}f_{0,2,4}-\frac{53}{270}f_{1,3,5}$ & & \\
         & $+\frac{18103}{39600}f_{0,0,6}+\frac{805}{3168}f_{2,0,6}+\frac{59}{13200}f_{3,0,6}-\frac{937}{3600}f_{1,0,8}$ & & \\
\hline 
\rule{0pt}{4ex} $(2,0,-1)$ & $ \frac{722869}{772200}(f_{1,0,0}+f_{3,0,0})+\frac{78797}{45900}f_{2,0,0}-\frac{15083}{43200}(f_{1,3,0}+f_{3,3,0})$ & 9 &  9.099 \\
					 & $+\frac{277}{2496}(f_{1,7,0}+f_{3,7,0})-\frac{473}{5712}f_{2,9,0}-\frac{4049}{2880}(f_{1,0,2}+f_{3,0,2})$ & &\\
					 & $+\frac{1859}{43200}(f_{1,3,3}+f_{3,3,3})+\frac{2749}{21600}f_{2,3,3}+\frac{853}{8640}(f_{1,2,4}+f_{3,2,4})$ & & \\
					 & $+\frac{3389}{30240}f_{2,2,4}-\frac{53}{270}f_{2,3,5}+\frac{3779}{10560}(f_{1,0,6}+f_{3,0,6})-\frac{937}{3600}f_{2,0,8}$ & & \\
\hline 
\rule{0pt}{4ex} $(1,1,-1)$ & $\frac{101}{2430}(f_{1,0,0}+f_{0,1,0})+\frac{12995}{4158}f_{1,1,0}+\frac{101}{34020}(f_{8,1,0}+f_{1,8,0})$ & 7& 9.386\\
           & $-\frac{538}{675}f_{0,0,2}-\frac{7}{54}(f_{2,0,2}+f_{0,2,2})-\frac{293}{2700}(f_{3,0,2}+f_{0,3,2})$      &  & \\
           & $-\frac{239}{144}f_{1,1,2}-\frac{7}{72}(f_{2,1,2}+f_{1,2,2})-\frac{199}{5400}f_{3,3,2}$ &  &   \\
           & $+\frac{641}{972}(f_{1,0,5}+f_{0,1,5})+\frac{641}{1944}(f_{2,1,5}+f_{1,2,5})-\frac{181}{176}f_{1,1,6}$     &  & \\
\hline 
\rule{0pt}{4ex} $(2,1,-1)$ & $ \frac{1}{156}(f_{1,0,0}+f_{3,0,0})+\frac{881}{1296}(f_{1,1,0}+f_{3,1,0})+\frac{81}{44}f_{2,1,0}$ & 7 &  9.386 \\
					 & $+\frac{1}{1872}(f_{1,7,0}+f_{3,7,0})-\frac{43}{72}(f_{1,0,2}+f_{3,0,2})$ & &\\
					 & $-\frac{119}{216}(f_{1,1,2}+f_{3,1,2})-\frac{13}{48}f_{2,1,2}-\frac{43}{144}(f_{1,2,2}+f_{3,2,2})$ & & \\
					 & $+\frac{7}{12}f_{2,0,5}+\frac{715}{1296}(f_{1,1,5}+f_{3,1,5})+\frac{7}{24}f_{2,2,5}-\frac{181}{176}f_{2,1,6}$ & & \\          
\hline 
\rule{0pt}{4ex} $(2,2,-1)$ & $\frac{1492}{663}f_{2,2,0}-\frac{5}{12}(f_{1,2,3}+f_{2,1,3}+f_{3,2,3}+f_{2,3,3})-\frac{19}{96}f_{2,2,3}$ & 10 & 5.561\\
                  & $+\frac{5}{24}(f_{1,2,7}+f_{2,1,7}+f_{3,2,7}+f_{2,3,7})+\frac{245}{1248}f_{2,2,7}-\frac{113}{272}f_{2,2,9}$      &  & \\
\hline 
\end{tabular}
	\label{tab-3}
\end{table}

\begin{table}[ht]
	\caption{Coefficient functionals $\lambda_{i,j,k}$ with the corresponding values of $n$ and $\|\sigma_{i,j,k}\|_1$ for $k=0$}
\renewcommand{\arraystretch}{1.3}
\centering\begin{tabular}{|c|l|c|c|}
\hline $(i,j,k)$   &   \multicolumn{1}{|c|}{$\lambda_{i,j,k} (f)$} & $n$ & $\Vert \sigma_{i,j,k}\Vert_1$  \\
\hline \multicolumn{4}{|c|}{$k=0$}\\
\hline 
\rule{0pt}{4ex} $(0,0,0)$ & $\frac{174511}{59400}f_{0,0,0}-\frac{1243}{1350}(f_{2,0,0}+f_{0,2,0}+f_{0,0,2})$ & 6 & 7.740\\
          & $-\frac{43}{990}(f_{6,0,0}+f_{0,6,0}+f_{0,0,6})+\frac{2987}{28800}(f_{2,3,0}+f_{3,2,0}+f_{3,0,2}$ & & \\
          & $+f_{0,3,2}+f_{2,0,3}+f_{0,2,3})-\frac{11}{75}(f_{3,3,0}+f_{3,0,3}+f_{0,3,3})$      &  & \\
          & $+\frac{259}{1920}(f_{4,1,0}+f_{1,4,0}+f_{4,0,1}+f_{0,4,1}+f_{1,0,4}+f_{0,1,4})-\frac{1}{27}f_{2,2,2}$ &  &   \\                  
\hline 
\rule{0pt}{4ex} $(1,0,0)$ & $\frac{92}{405}f_{0,0,0}+\frac{1301}{432}f_{1,0,0}+\frac{13}{108}f_{2,0,0}+\frac{5}{1296}f_{5,0,0}$ & 4 & 7.649\\
          & $-\frac{155}{216}(f_{0,1,0}+f_{0,0,1})-\frac{1}{36}(f_{0,2,0}+f_{0,0,2})$ &  & \\
          & $-\frac{25}{54}(f_{1,2,0}+f_{1,0,2})-\frac{41}{108}(f_{2,1,0}+f_{2,0,1})$ &  & \\
          & $+\frac{23}{360}(f_{0,3,0}+f_{0,0,3})+\frac{1}{36}(f_{2,3,0}+f_{2,0,3})$ & & \\
          & $+\frac{7}{27}(f_{0,2,1}+f_{0,1,2})+\frac{7}{54}(f_{2,2,1}+f_{2,1,2})-\frac{4}{27}f_{1,2,2}$ &  &   \\ 
\hline 
\rule{0pt}{4ex} $(2,0,0)$& $\frac{106}{495}f_{0,0,0}+\frac{1115}{432}f_{2,0,0}+\frac{101}{180}f_{3,0,0}+\frac{53}{7920}f_{6,0,0}$ & 4 & 7.649\\
          & $-\frac{61}{144}(f_{1,1,0}+f_{1,0,1})-\frac{97}{144}(f_{3,1,0}+f_{3,0,1})$ &  & \\
          & $-\frac{1}{6}(f_{1,2,0}+f_{1,0,2})-\frac{35}{108}(f_{2,2,0}+f_{2,0,2})$ &  & \\
          & $+\frac{17}{240}(f_{1,3,0}+f_{1,0,3})+\frac{1}{48}(f_{3,3,0}+f_{3,0,3})$ & & \\
          & $+\frac{7}{36}(f_{1,2,1}+f_{3,2,1}+f_{1,1,2}+f_{3,1,2})-\frac{4}{27}f_{2,2,2}$ &  &   \\         
\hline 
\rule{0pt}{4ex} $(3,0,0)$& $\frac{697}{180}f_{3,0,0}+\frac{1}{24}(f_{2,0,0}+f_{4,0,0})$ & 3 & 9.945\\
          & $-\frac{11}{24}(f_{2,1,0}+f_{2,0,1}+f_{4,1,0}+f_{4,0,1})-\frac{77}{72}(f_{3,1,0}+f_{3,0,1})$ &  & \\
          & $-\frac{7}{36,}(f_{3,2,0}+f_{3,0,2})+\frac{11}{120}(f_{3,3,0}+f_{3,0,3})$ &  & \\
          & $+\frac{2}{3}(f_{2,1,1}+f_{4,1,1})-\frac{1}{18}(f_{3,2,1}+f_{3,1,2})$ & & \\
\hline 
\rule{0pt}{4ex} $(1,1,0)$& $-\frac{16}{33}f_{0,0,0}-\frac{14}{99}(f_{3,0,0}+f_{0,3,0})+\frac{38}{15}f_{1,1,0}+\frac{1}{11}(f_{2,1,0}+f_{1,2,0})$ & 3 & 5.508\\
                  & $-\frac{4}{99}(f_{3,2,0}+f_{2,3,0}+f_{2,0,1}+f_{0,2,1})-\frac{23}{88}f_{1,1,1}$ & & \\
                  & $-\frac{17}{44}(f_{2,1,1}+f_{1,2,1})+\frac{59}{264}(f_{3,1,1}+f_{1,3,1})-\frac{4}{99}f_{2,2,1}$      &  & \\
                  & $-\frac{1}{4}f_{1,1,2}+\frac{11}{120}f_{1,1,3}$ &  &   \\
\hline
\end{tabular}
	\label{tab-4}
\end{table}

\begin{table}[ht]
	\caption{Coefficient functionals $\lambda_{i,j,k}$ with the corresponding values of $n$ and $\|\sigma_{i,j,k}\|_1$ for $k=0$}
\renewcommand{\arraystretch}{1.3}
\centering\begin{tabular}{|c|l|c|c|}
\hline $(i,j,k)$   &   \multicolumn{1}{|c|}{$\lambda_{i,j,k} (f)$} & $n$ & $\Vert \sigma_{i,j,k}\Vert_1$  \\
\hline \multicolumn{4}{|c|}{$k=0$}\\
\hline
\rule{0pt}{4ex} $(2,1,0)$  & $-\frac{188}{945}f_{0,0,0}-\frac{8}{63}f_{4,0,0}+\frac{37}{405}f_{0,1,0}+\frac{1043}{540}f_{2,1,0}+\frac{11}{360}f_{3,1,0}$ & 3 & 5.108 \\
           & $+\frac{5}{648}f_{5,1,0}+\frac{43}{108}f_{2,2,0}+\frac{1}{36}f_{4,2,0}-\frac{5}{36}f_{1,3,0}-\frac{7}{45}f_{3,3,0}-\frac{46}{135}f_{2,0,1}$  &  &  \\
			  	 & $+\frac{5}{63}f_{0,1,1}+\frac{5}{84}f_{4,1,1}-\frac{91}{108}f_{2,2,1}+\frac{121}{360}f_{2,3,1}-\frac{1}{4}f_{2,1,2}+\frac{11}{120}f_{2,1,3}$ & & \\
\hline 
\rule{0pt}{4ex} $(3,1,0)$  & $-\frac{29}{216}(f_{1,0,0}+f_{5,0,0})-\frac{41}{1080}(f_{2,0,0}+f_{4,0,0})+\frac{29}{384}(f_{1,1,0}+f_{5,1,0})$ & 3 & 5.048 \\
           & $+\frac{1867}{960}f_{3,1,0}+\frac{29}{72}f_{3,2,0}-\frac{23}{160}(f_{2,3,0}+f_{4,3,0})-\frac{29}{90}f_{3,0,1}$  &  &  \\
			  	 & $+\frac{5}{96}(f_{1,1,1}+f_{5,1,1})-\frac{59}{72}f_{3,2,1}+\frac{79}{240}f_{3,3,1}-\frac{1}{4}f_{3,1,2}+\frac{11}{120}f_{3,1,3}$ & & \\
\hline  
\rule{0pt}{4ex} $(2,2,0)$ & $\frac{358}{165}f_{2,2,0}-\frac{10}{231}(f_{1,0,0}+f_{3,0,0}+f_{0,1,0}+f_{0,3,0})$ & 3  & 4.129 \\
                  & $-\frac{5}{154}(f_{4,1,0}+f_{4,3,0}+f_{1,4,0}+f_{3,4,0})-\frac{167}{264}f_{2,2,1}$      &  & \\
                  & $-\frac{5}{66}(f_{3,2,1}+f_{2,3,1})+\frac{5}{132}(f_{4,2,1}+f_{2,4,1})-\frac{21}{44}f_{2,2,2}$ & & \\
                  & $+\frac{5}{88}(f_{1,2,2}+f_{2,1,2}+f_{3,2,2}+f_{2,3,2})+\frac{11}{120}f_{2,2,3}$ &  &   \\
\hline 
\rule{0pt}{4ex} $(3,2,0)$  & $-\frac{460}{12033}f_{2,0,0}-\frac{860}{12033}f_{4,0,0}-\frac{25}{1146}(f_{1,1,0}+f_{5,3,0})$ & 3 & 4.028\\
           & $-\frac{135}{2674}f_{2,4,0}+\frac{6098}{2865}f_{3,2,0}-\frac{175}{18909}f_{6,2,0}-\frac{320}{18909}f_{0,2,0}$  &  &  \\
			  	 & $-\frac{85}{2674}f_{4,4,0}-\frac{1009}{1528}f_{3,2,1}-\frac{55}{573}f_{3,3,1}+\frac{55}{1146}f_{3,4,1}-\frac{3409}{6876}f_{3,2,2}$ & & \\
			  	 & $+\frac{245}{4584}(f_{3,1,2}+f_{3,3,2})+\frac{5}{72}(f_{2,2,2}+f_{4,2,2})+\frac{11}{120}f_{3,2,3}$      &  & \\
\hline 
\rule{0pt}{4ex} $(4,2,0)$  & $-\frac{5}{84}(f_{3,0,0}+f_{5,0,0})+\frac{193}{90}f_{4,2,0}-\frac{5}{144}(f_{1,2,0}+f_{7,2,0})$ & 3 & 3.994 \\
           & $-\frac{5}{112}(f_{3,4,0}+f_{5,4,0})-\frac{73}{96}f_{4,2,1}+\frac{5}{96}(f_{2,2,1}+f_{4,4,1}$ &  &  \\
           & $+f_{6,2,1}+f_{4,1,2}+f_{4,3,2})-\frac{5}{48}f_{4,3,1}-\frac{17}{48}f_{4,2,2}+\frac{11}{120}f_{4,2,3}$ & & \\
\hline 
\rule{0pt}{4ex} $(3,3,0)$& $\frac{85}{54}f_{3,3,0}-\frac{55}{1536}(f_{1,1,1}+f_{5,1,1}+f_{1,5,1}+f_{5,5,1})$ & 5 & 2.617\\
         & $-\frac{85}{144}f_{3,3,2}+\frac{5}{768}(f_{3,1,3}+f_{1,3,3}+f_{5,3,3}+f_{3,5,3})$ & & \\
         & $+\frac{5}{96}(f_{3,2,4}+f_{2,3,4}+f_{4,3,4}+f_{3,4,4})-\frac{259}{3456}f_{3,3,5}$ &  &   \\
\hline 
\end{tabular}
	\label{tab1}
\end{table}

\begin{table}[ht]
	\caption{Coefficient functionals $\lambda_{i,j,k}$ with the corresponding values of $n$ and $\|\sigma_{i,j,k}\|_1$ for $k=1,2,3$}
\renewcommand{\arraystretch}{1.3}
\centering\begin{tabular}{|c|l|c|c|}
\hline $(i,j,k)$   &   \multicolumn{1}{|c|}{$\lambda_{i,j,k} (f)$} & $n$ & $\Vert \sigma_{i,j,k}\Vert_1$  \\
\hline \multicolumn{4}{|c|}{$k=1$}\\
\hline 
\rule{0pt}{4ex} $(1,1,1)$& $\frac{41}{96}f_{1,1,1}+\frac{5}{18}(f_{2,1,1}+f_{1,2,1}+f_{1,1,2})$ & 6 & 1.730 \\
					 & $-\frac{35}{288}(f_{5,1,1}+f_{1,5,1}+f_{1,1,5})+\frac{5}{144}(f_{7,1,1}+f_{1,7,1}+f_{1,1,7})$& & \\
\hline 
\rule{0pt}{4ex} $(2,1,1)$ & $-\frac{4}{9}f_{2,0,0}-\frac{2}{9}(f_{2,2,0}+f_{2,0,2})-\frac{2}{21}f_{0,1,1}+\frac{55}{24}f_{2,1,1}-\frac{5}{168}f_{4,1,1}$ & 2 & 3.75\\
          & $-\frac{1}{12}(f_{3,1,1}+f_{2,2,1}+f_{2,1,2})+\frac{1}{24}(f_{2,3,1}+f_{2,1,3})-\frac{1}{9}f_{2,2,2}$      &  & \\
\hline 
\rule{0pt}{4ex} $(3,1,1)$ & $-\frac{4}{9}f_{3,0,0}-\frac{2}{9}(f_{3,2,0}+f_{3,0,2})-\frac{5}{96}(f_{1,1,1}+f_{5,1,1})+\frac{35}{16}f_{3,1,1}$ & 2 & 3.542\\
                  & $-\frac{1}{12}(f_{3,2,1}+f_{3,1,2})+\frac{1}{24}(f_{3,3,1}+f_{3,1,3})-\frac{1}{9}f_{3,2,2}$      &  & \\
\hline 
\rule{0pt}{4ex} $(2,2,1)$ & $-\frac{1}{7}f_{2,2,0}-\frac{1}{12}(f_{1,1,0}+f_{3,1,0}+f_{1,3,0}+f_{3,3,0})+\frac{13}{8}f_{2,2,1}$ & 3 & 2.370\\
          & $-\frac{1}{24}(f_{2,1,3}+f_{1,2,3}+f_{2,2,3}+f_{3,2,3}+f_{2,3,3})+\frac{5}{84}f_{2,2,4}$   &  & \\ 
\hline  \multicolumn{4}{|c|}{$k=2$}\\
\hline 
\rule{0pt}{4ex} $(2,2,2)$ & $\frac{9}{4}f_{2,2,2}-\frac{5}{24}(f_{2,1,2}+f_{2,3,2}+f_{1,2,2}+f_{3,2,2}+f_{2,2,1}+f_{2,2,3})$ & 1 & 3.5\\
\hline  \multicolumn{4}{|c|}{$k=3$}\\
\hline 
\rule{0pt}{4ex} $(3,3,3)$ & $\frac{21}{16}f_{3,3,3}-\frac{5}{96}(f_{3,1,3}+f_{3,5,3}+f_{1,3,3}+f_{5,3,3}+f_{3,3,1}+f_{3,3,5})$ & 2 & 1.625\\
\hline 
\end{tabular}
	\label{tab12}
\end{table}

\bigskip

Now we give an upper bound for the infinity norm of the proposed quasi-interpol\-ating operator and error estimates for it.

\begin{theorem}\label{tnorm_tri} For the operator $Q$, the following bound holds
\begin{equation}\label{nQ}
	\Vert Q \Vert_\infty \leq  9.945 .
\end{equation}
\end{theorem}

\begin{proof} For $\Vert f \Vert_\infty \leq 1$, taking into account (\ref{nn}) and Tables \ref{tab-3} through \ref{tab12}, we get an upper bound for $\Vert Q \Vert_\infty$ by $\Vert \sigma_{3,0,0}\Vert_1$. Therefore, we obtain (\ref{nQ}). 
\end{proof}

Standard results in approximation theory \cite[Chap.3]{dhr} allow us to immediately deduce the following theorem.

\begin{theorem}\label{err_tri1} Let $f \in C^4(\Omega)$ and $\vert \gamma \vert=0,1,2,3$ with $\gamma=(\gamma_1,\gamma_2,\gamma_3)$, $\vert \gamma \vert=\gamma_1+\gamma_2+\gamma_3$. Then there exist constants $K_\gamma>0$, independent on $h$, such that
$$
\left\|D^\gamma(f-Qf)\right\|_\infty \leq K_\gamma h^{4-\vert \gamma \vert} \max_{\vert \beta \vert =4} \left\|D^{ \beta }f\right\|_\infty,
$$
where  $D^{\beta}=D^{\beta_1\beta_2\beta_3} =\frac{\partial^{\vert \beta \vert}}{\partial x^{\beta_1}\partial y^{\beta_2}\partial z^{\beta_3}}$.
\end{theorem}

\clearpage

\section{Numerical results} \label{num}

In order to illustrate both the approximation properties and the visual quality of our $C^2$ trivariate splines, in this section we present some numerical results obtained by computational procedures developed in the Matlab environment. For the evaluation of box splines we use the algorithm based on the Bernstein-B\'ezier form (BB-form) of the box spline, proposed in \cite{kp}.

In order to explore the volumetric data, we visualize some isosurfaces of our trivariate quasi-interpolating splines, generated as a very fine triangular mesh by using the Matlab procedure {\tt isosurface}, with the {\tt lightning phong} command for surface shading \cite{Mat}.

We want to compare our method with other ones proposed in the literature for bounded domains. To this aim, we consider the papers \cite{nrsz,r1,SZ} and the same test functions there presented, i.e.
\begin{itemize}  
\item[-] the Marschner-Lobb function
$$
\begin{array}{l}
f_1(x,y,z)=\\
\displaystyle\frac{1}{2(1+\beta_1)}\left(1- \sin \frac{\pi z}{2} + \beta_1\left(1+ \cos \left(2\pi \beta_2 \cos\left(\frac{\pi \sqrt{x^2+y^2}}{2}\right)\right)\right) \right)
\end{array}
$$
with $\beta_1=\frac14$ and $\beta_2=6$ on $\Omega=[-1,1]^3$. This function is extremely oscillating and therefore it represents a difficult test for any efficient three-dimensional reconstruction method;
\item[-] the smooth trivariate test function of Franke type
$$
\begin{array}{ll}
f_2(x,y,z)= & \displaystyle\frac12 e^{-10((x-\frac14)^2+(y-\frac14)^2)} + \frac34 e^{-16((x-\frac12)^2+(y-\frac14)^2+(z-\frac14)^2)} \\
           & +\displaystyle\frac12 e^{-10((x-\frac34)^2+(y-\frac18)^2+(z-\frac12)^2)} - \frac14 e^{-20((x-\frac34)^2+(y-\frac34)^2)}
\end{array}
$$
on $\Omega=[0,1]^3$;
\item[-] $f_3(x,y,z)= \frac19 \tanh (9(z-x-y)+1)$ on $\Omega=[-\frac12,\frac12]^3$.
\end{itemize}

We recall that N\"urnberger et al. \cite{nrsz} propose quasi-interpolation by quadratic $C^1$ piecewise polynomials in three variables in BB-form, with approximation order two. We denote their quasi-interpolating operator by $S_q$. Sorokina and Zeilfelder \cite{SZ} present local quasi-interpolation by cubic $C^1$ splines in BB-form, with approximation order two. We denote such a QI by $S_c$.

Remogna and Sablonni\`ere \cite{r1} propose a near-best quasi-interpolant, with approximation order three, defined as blending sum of univariate and bivariate $C^1$ quadr\-atic spline quasi-interpolants, on a partition of $\Omega$ into vertical prisms with triangular sections. We denote it by $R_1$.

\subsection{Errors and convergence orders}
We construct our quasi-interpolant $Qf$ with $\lambda_\alpha(f)$ given in Tables \ref{tab-3} through \ref{tab12} and with approximation order four. Here we assume $m_1=m_2=m_3=m$ and $m=16,32,64,128$. For the evaluation of $Qf$, we consider a $139 \times 139 \times 139$ uniform three-dimensional grid $G$ of points in the domain and, for each test function, we compute the maximum absolute error 
$$
E_Sf:=\max_{(u,v,w) \in G}\vert f(u,v,w)-Sf(u,v,w) \vert,
$$
with $S=Q$, $S_q$, $S_c$, $R_1$. 

We report the results in Table \ref{tab:Err3} and Figs. \ref{log1}, \ref{log3}. In Table \ref{tab:Err3} we also report an estimate of the approximation order, $rf$, obtained by the logarithm to base two of the ratio between two consecutive errors. 

\begin{table}[ht]
    \caption{Maximum absolute errors and numerical convergence orders}
\centering\begin{tabular}{|c|cc|cc|cc|cc|}
\hline
$m$ & $E_Qf$  & $rf$  & $E_{S_q}f$  & $rf$ & $E_{S_c}f$ & $rf$ & $E_{R_1}f$ & $rf$ \\
\hline\multicolumn{9}{|c|}{$f_1$}\\
\hline $16$ & 2.0(-1) &     &         &     &         &     & 1.9(-1) &    \\
       $32$ & 1.3(-1) & 0.6 & 1.8(-1) &     & 1.8(-1) &     & 1.5(-1) & 0.4 \\
       $64$ & 6.5(-2) & 1.0 & 1.2(-1) & 0.5 & 1.2(-1) & 0.5 & 3.2(-2) & 2.2 \\
      $128$ & 2.1(-2) & 1.7 & 4.0(-2) & 1.6 & 4.0(-2) & 1.6 & 4.6(-3) & 2.8 \\
\hline
\hline \multicolumn{9}{|c|}{$f_2$}\\
\hline $16$ & 1.7(-2) &     & 4.3(-2) &     & 4.3(-2) &     & 6.6(-3) &    \\
       $32$ & 8.0(-4) & 4.4 & 1.1(-2) & 2.0 & 1.1(-2) & 2.0 & 8.2(-4) & 3.0  \\
       $64$ & 5.2(-5) & 3.9 & 2.8(-3) & 2.0 & 2.8(-3) & 2.0 & 9.8(-5) & 3.1 \\
      $128$ & 3.3(-6) & 4.0 & 6.9(-4) & 2.0 & 6.9(-4) & 2.0 & 8.5(-6) & 3.5 \\
\hline
\hline \multicolumn{9}{|c|}{$f_3$}\\
\hline $16$ & 6.2(-3) &     & 8.8(-3) &     & &  & 6.2(-3) &     \\
       $32$ & 8.2(-4) & 2.9 & 2.4(-3) & 1.9 & &  & 1.1(-3) & 2.5 \\
       $64$ & 8.9(-5) & 3.2 & 6.3(-4) & 1.9 & &  & 1.7(-4) & 2.7 \\
      $128$ & 7.9(-6) & 3.5 & 1.6(-4) & 2.0 & &  & 1.7(-5) & 3.3 \\
\hline
\end{tabular}
    \label{tab:Err3}
\end{table}

\begin{figure}[ht]
\begin{minipage}{60mm}
\centering\includegraphics[width=5.5cm]{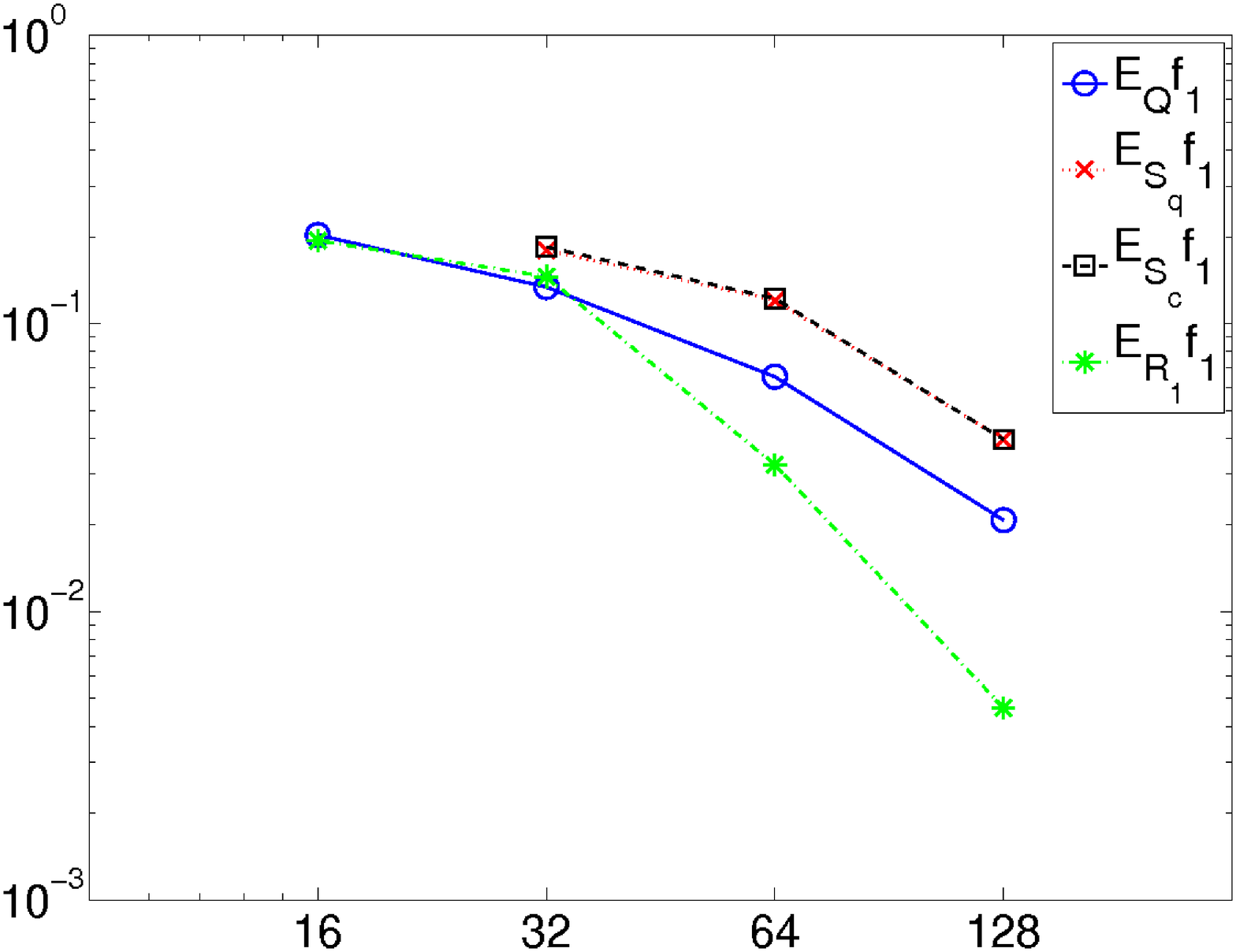}
\centerline{$(a)$}
\end{minipage}
\hfil
\begin{minipage}{60mm}
\centering\includegraphics[width=5.5cm]{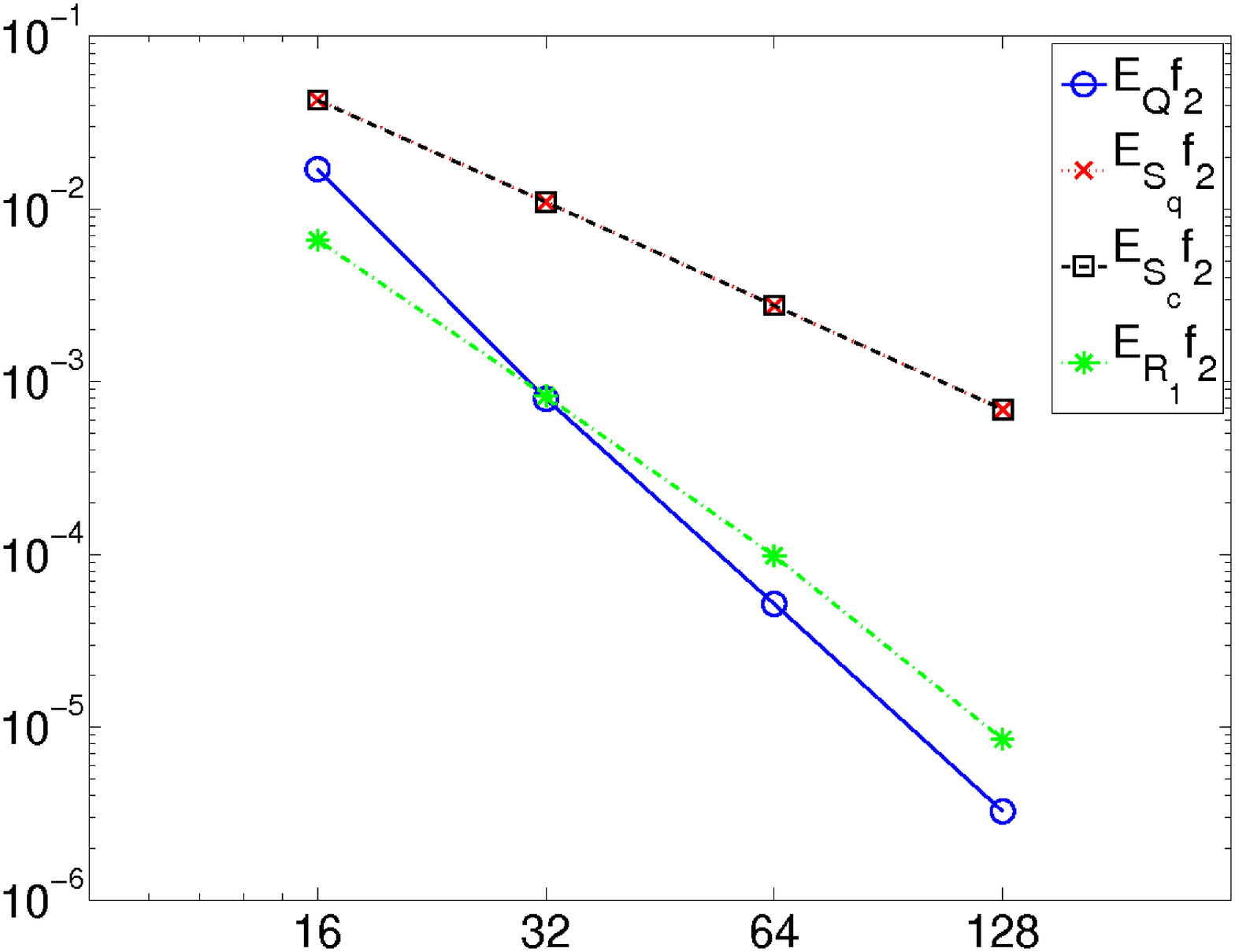}
\centerline{$(b)$}
\end{minipage}
\caption{Absolute errors for $(a)$ $f_1$ and $(b)$ $f_2$ versus interval number per side}
	\label{log1}
\end{figure}

\begin{figure}[ht]
\centering\includegraphics[width=5.5cm]{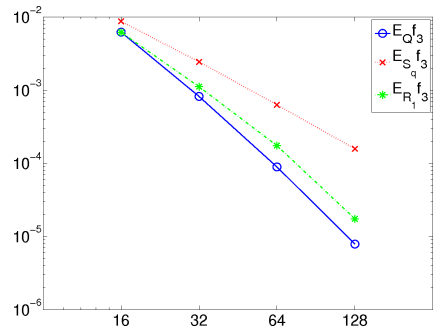}
\caption{Absolute errors for $f_3$ versus interval number per side}
	\label{log3}
\end{figure}

Comparing the results, as we can expect, we notice that using our $C^2$ quartic splines, due to the higher approximation order, the error decreases faster than using $S_q$, $S_c$ and $R_1$, except in case of function $f_1$ for which the errors are comparable.

\subsection{Isosurfaces}
In order to explore the volumetric data and to show the good approximation properties of the operator $Q$, here proposed, we visualize some isosurfaces of the trivariate $C^2$ splines approximating the considered test functions.

In Fig. \ref{ml}$(a)$, we show the isosurface obtained from $f_1$ with isovalue $\rho = 1/2$, i.e. a view on the set of all points $(x,y,z) \in [-1, 1]^3$, such that $f_1(x,y,z) = \rho$. In Fig. \ref{ml}$(b)$, we show the isosurface of the approximating spline $Qf_1$ with isovalue $\rho = 1/2$, with $m=64$ and the maximal error is color coded from red $\approx 6.5 \cdot 10^{-2}$ to blue $=0$.

Moreover, a visualization of the spline $Qf_2$, for $m=32$, using isovalues $\rho=-0.1$, 0, 0.1, 0.2, 0.5 and 0.8, is shown in Fig. \ref{iso2}, where we color coded the maximal error from red to blue (see the colorbar in each figure). For a qualitative visual comparison with competing methods, we can refer to the corresponding isosurfaces shown in \cite{nrsz}.

\begin{figure}[ht]
\begin{minipage}{60mm}
\centering\includegraphics[width=5.5cm]{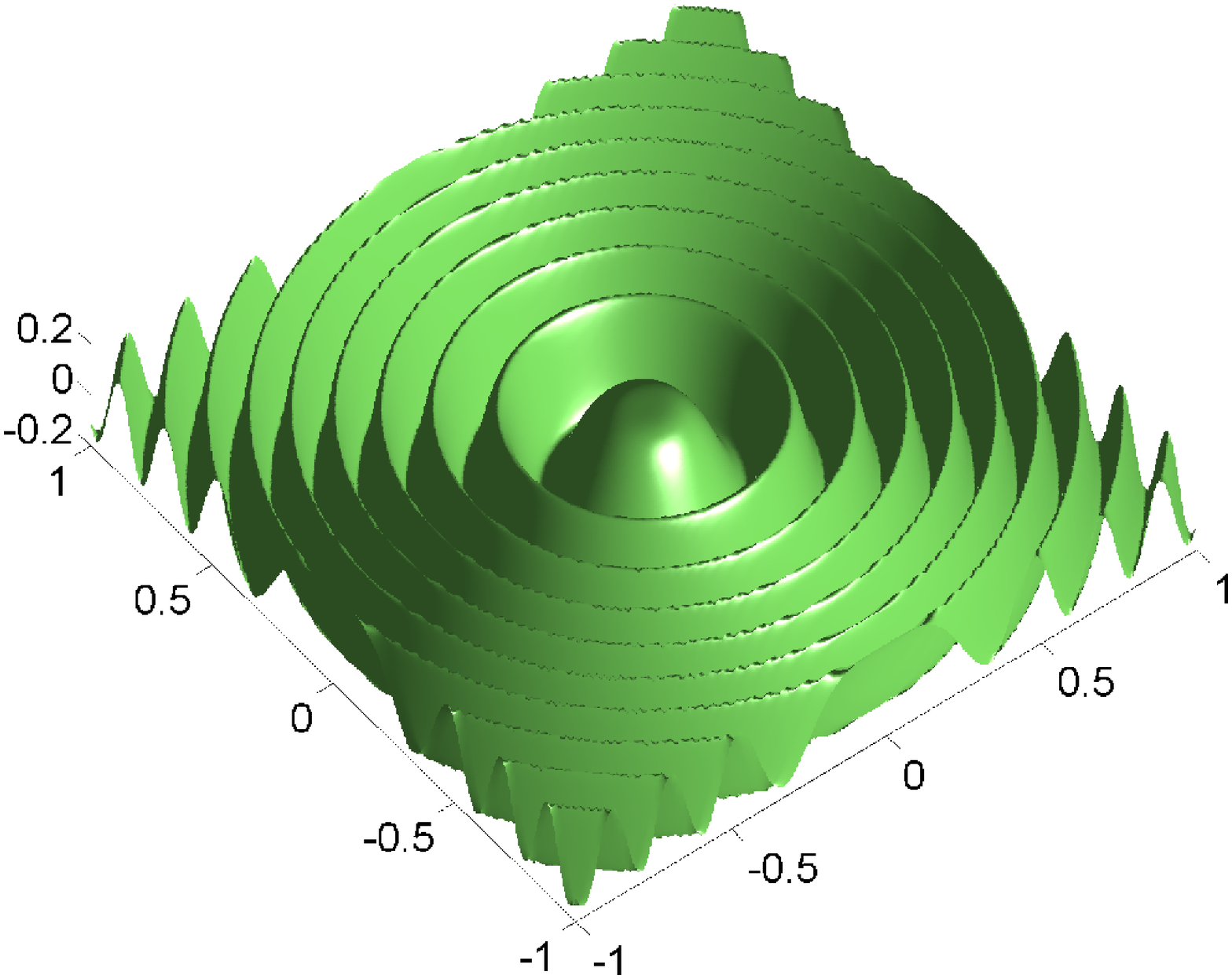}
\centerline{$(a)$}
\end{minipage}
\hfil
\begin{minipage}{60mm}
\centering\includegraphics[width=6.1cm]{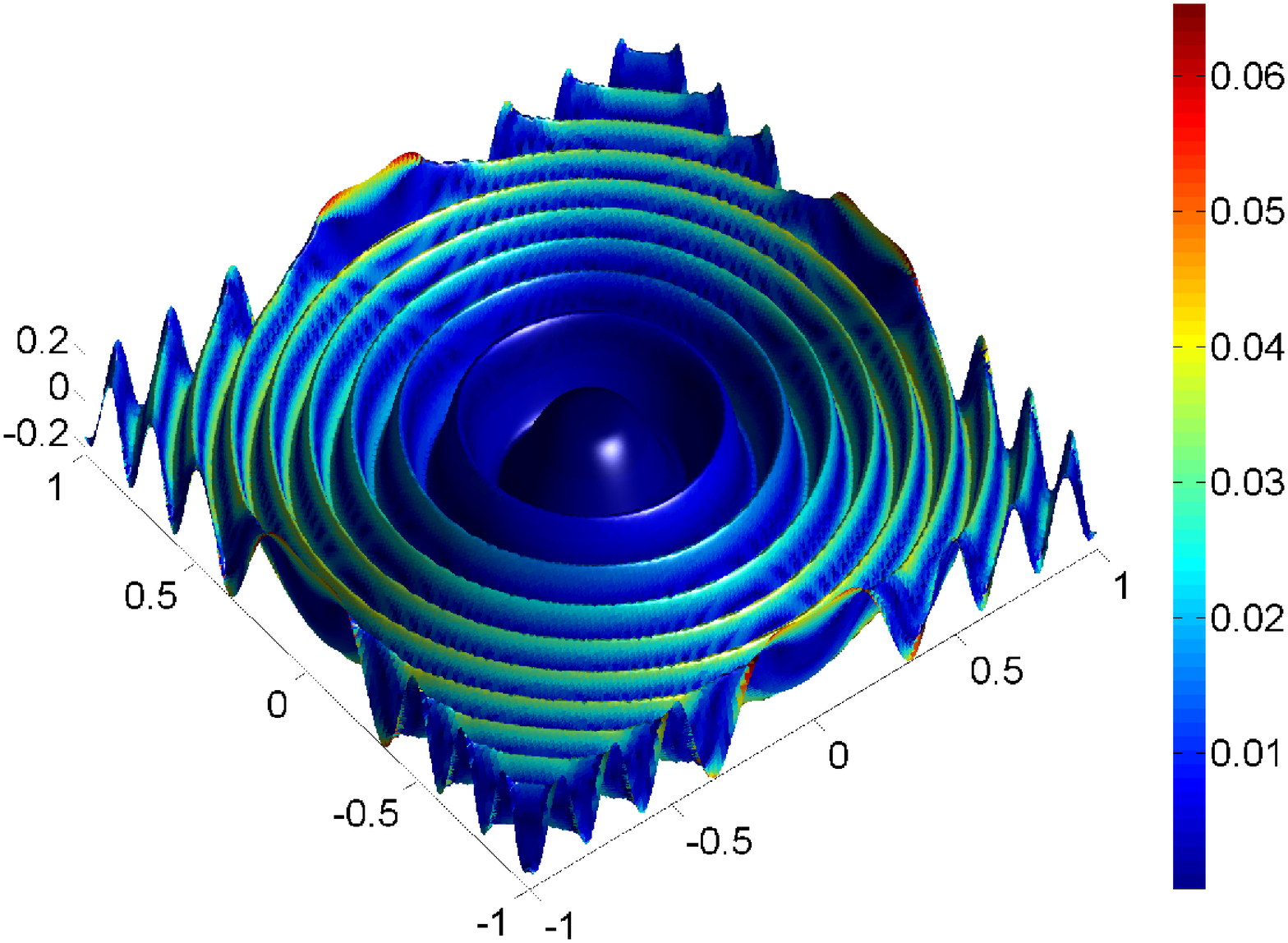}
\centerline{$(b)$}
\end{minipage}
\caption{For the isovalue $\rho = 1/2$ the isosurface of $(a)$ $f_1$ and $(b)$ $Qf_1$, with $m=64$}
	\label{ml}
\end{figure}

\begin{figure}[ht]
\begin{minipage}{60mm}
\centering\includegraphics[width=6cm]{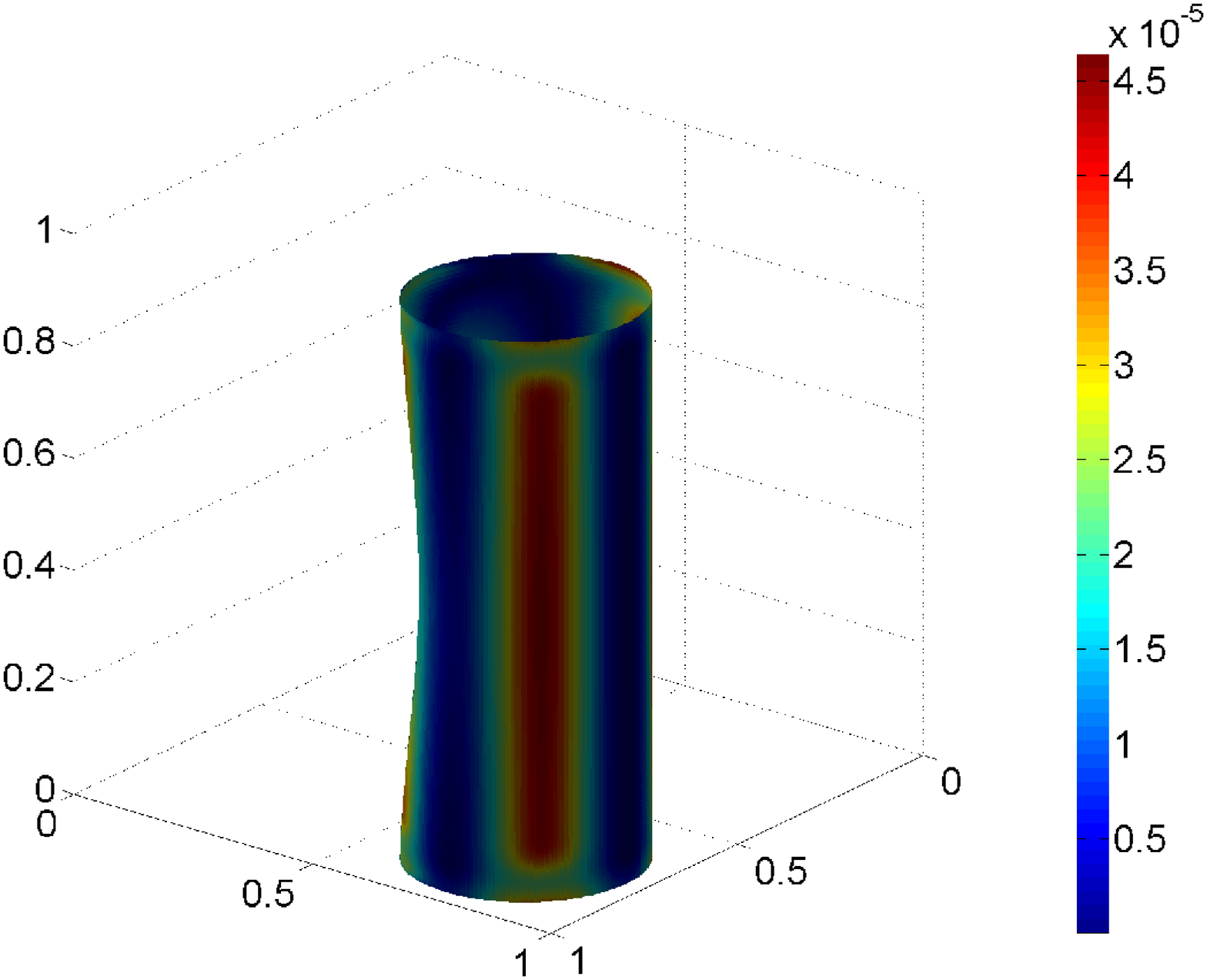}
\centerline{$(a)$}
\end{minipage}
\hfil
\begin{minipage}{60mm}
\centering\includegraphics[width=6cm]{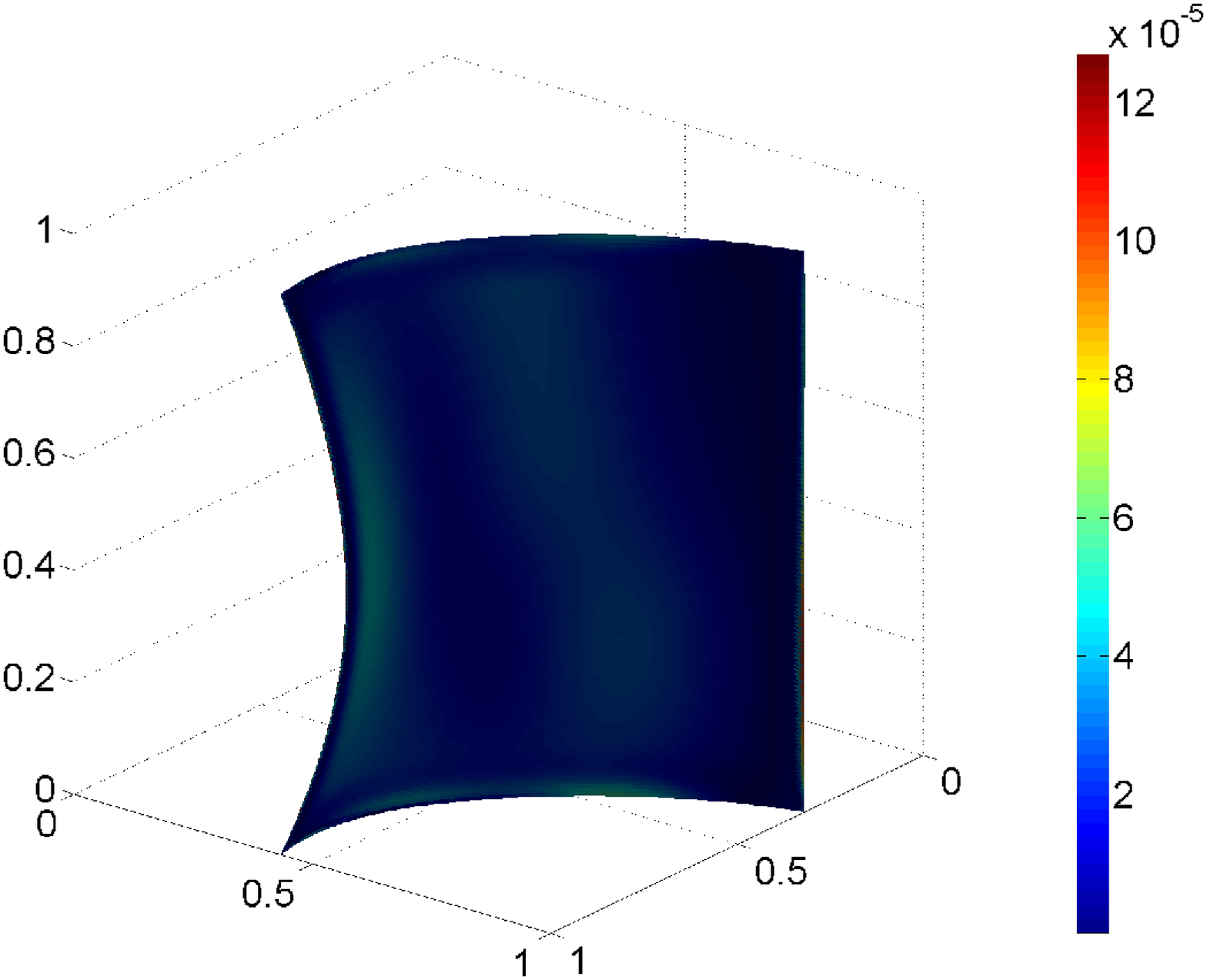}
\centerline{$(b)$}
\end{minipage}
\hfil\begin{minipage}{60mm}
\centering\includegraphics[width=6cm]{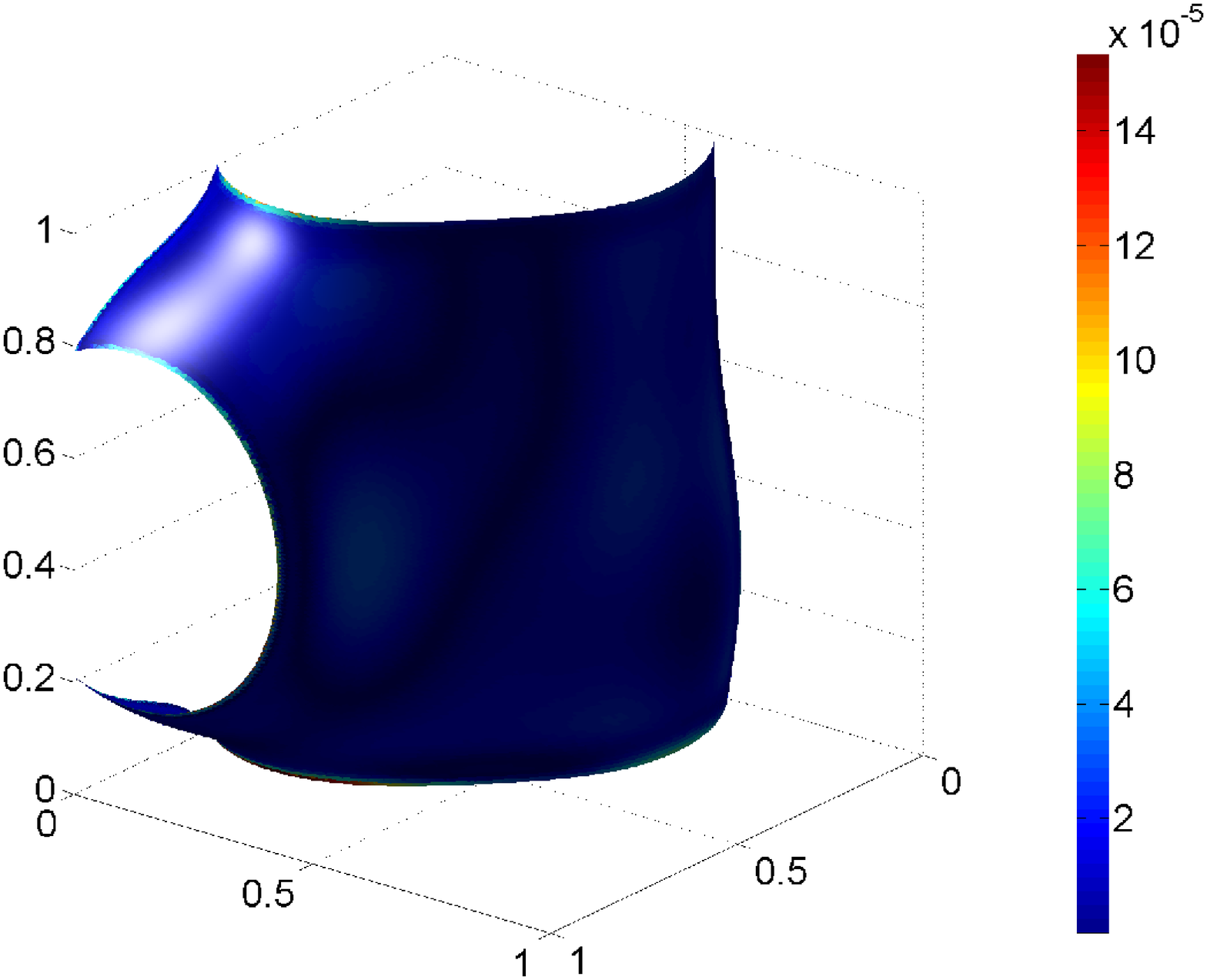}
\centerline{$(c)$}
\end{minipage}
\hfil
\begin{minipage}{60mm}
\centering\includegraphics[width=6cm]{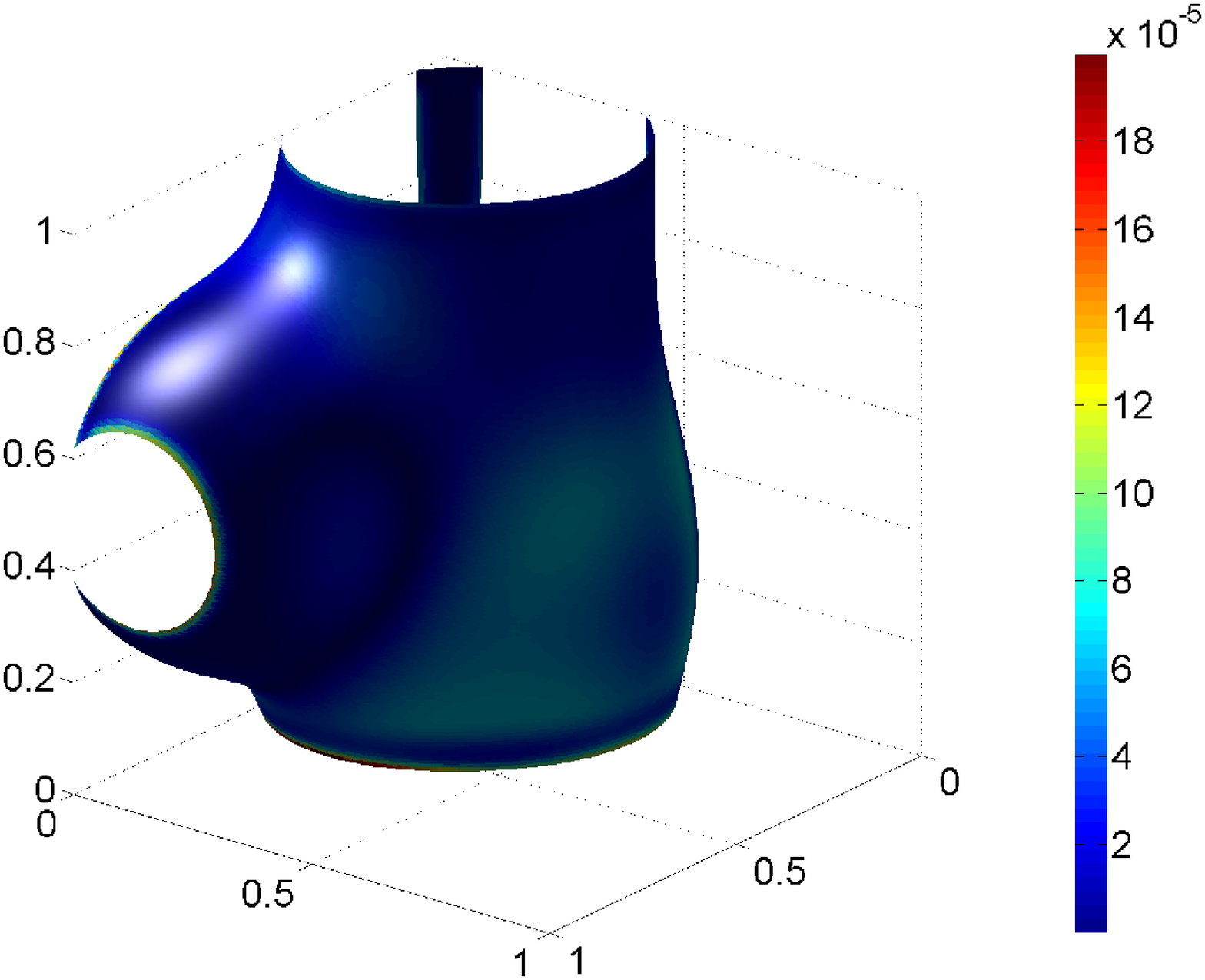}
\centerline{$(d)$}
\end{minipage}
\hfil
\begin{minipage}{60mm}
\centering\includegraphics[width=6cm]{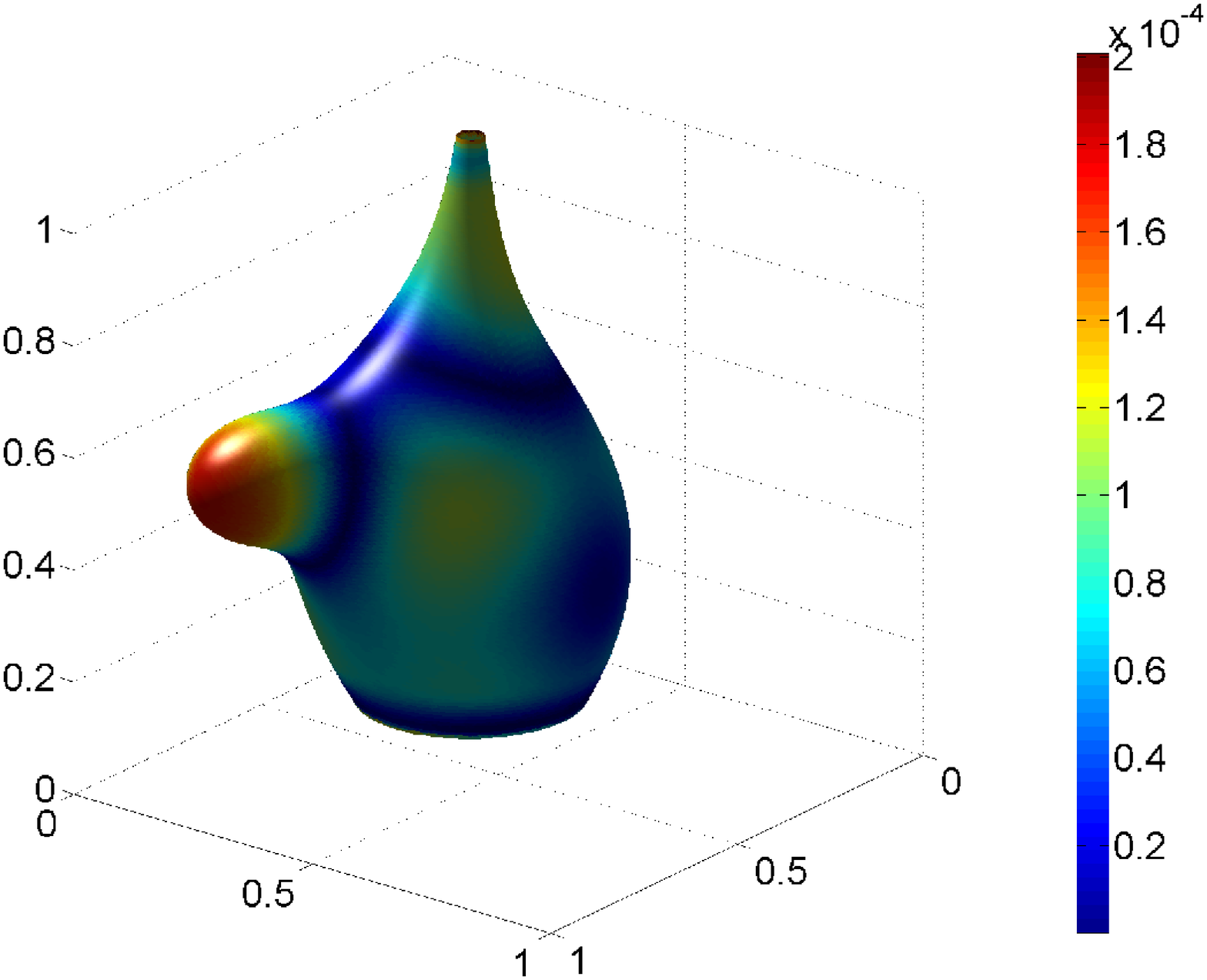}
\centerline{$(e)$}
\end{minipage}
\hfil
\begin{minipage}{60mm}
\centering\includegraphics[width=6cm]{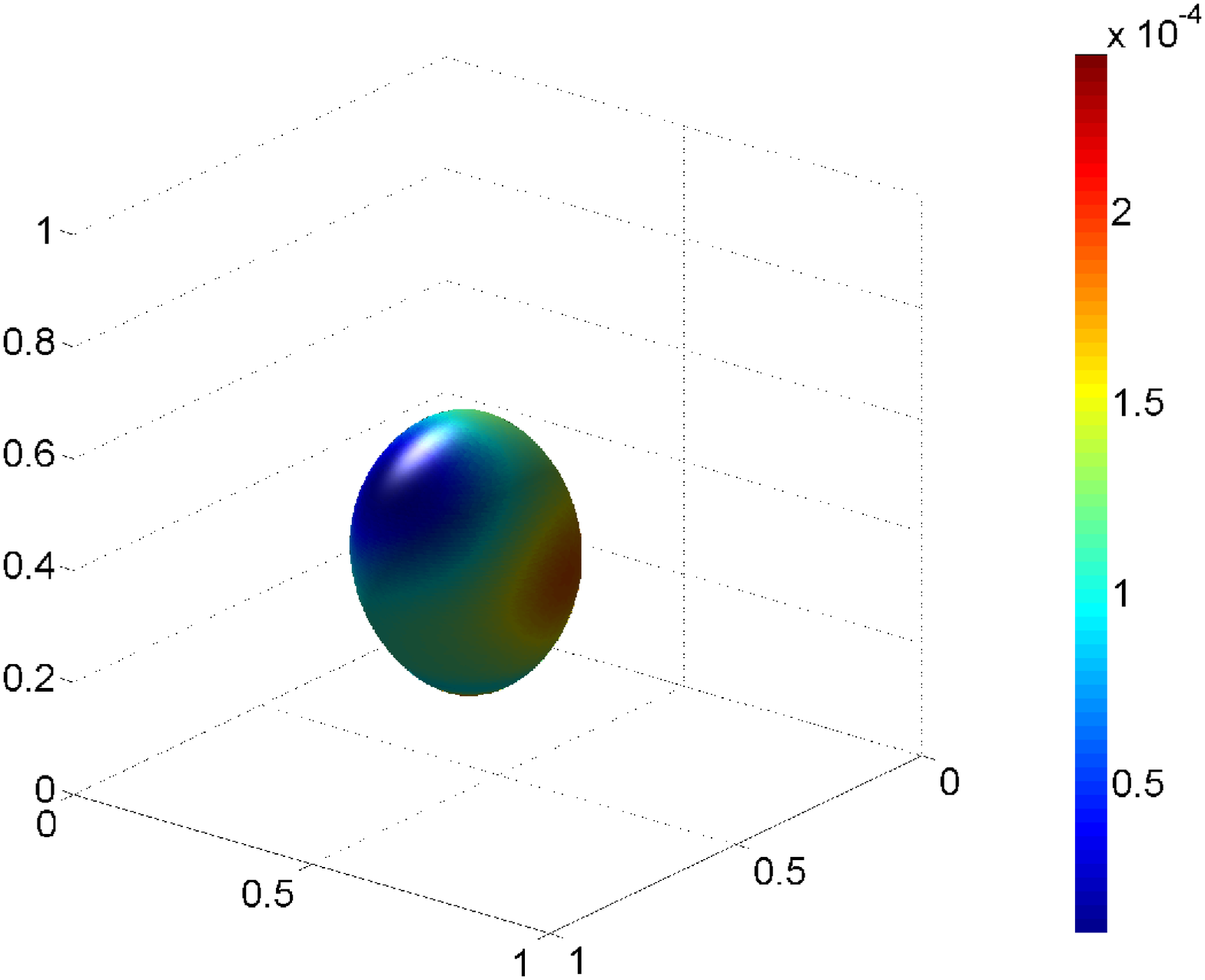}
\centerline{$(f)$}
\end{minipage}
\caption{Isosurfaces of $Qf_2$ for $m=32$, with isovalues $(a)$ $\rho =-0.1$, $(b)$ $\rho =0$, $(c)$ $\rho =0.1$, $(d)$ $\rho =0.2$, $(e)$ $\rho =0.5$ and $(f)$ $\rho =0.8$}
	\label{iso2}
\end{figure}

\subsection{Applications to real world data}
Finally, we present examples based on real world data that can be considered as typical for many applications, where a precise evaluation and a high visual quality are the goals of visualization. In particular, starting from a discrete set of data, we obtain a non-discrete model of the real object with $C^2$ smoothness.

In Fig. \ref{rw} we show two isosurfaces of the $C^2$ quartic spline $Qf$, resulting from the application of our method in the approximation of a gridded volume data set consisting of $256 \times 256 \times 99$ data samples, obtained from a CT scan of a cadaver head (courtesy of University of North Carolina). In order to visualize the isosurfaces, corresponding to the isovalues $\rho=60$, 90, we evaluate the spline at $N\approx 8,6 \times 10^6$ points.

In Fig. \ref{br} we show the isosurface of the $C^2$ quartic spline $Qf$ that approximates a gridded volume data set of $256 \times 256 \times 99$ data samples, obtained from a MR study of head with skull partially removed to reveal brain (courtesy of University of North Carolina). Also in this case, in order to visualize the isosurface, corresponding to the isovalue $\rho=40$, we evaluate the spline at $N\approx 8,6 \times 10^6$ points.

\begin{figure}[ht]
\begin{minipage}{60mm}
\centering\includegraphics[width=6.25cm]{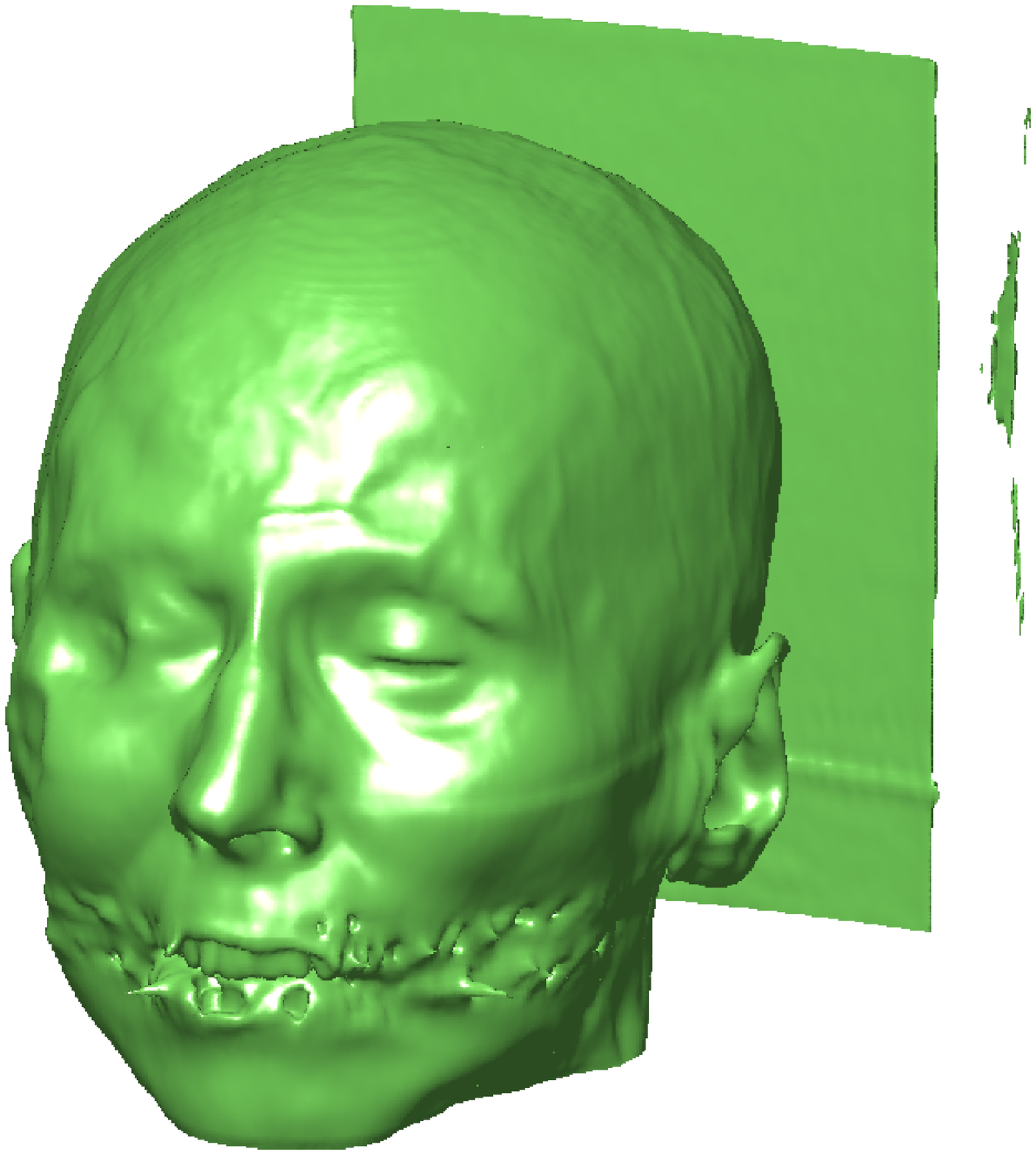}
\centerline{$(a)$}
\end{minipage}
\hfil
\begin{minipage}{60mm}
\centering\includegraphics[width=5.5cm]{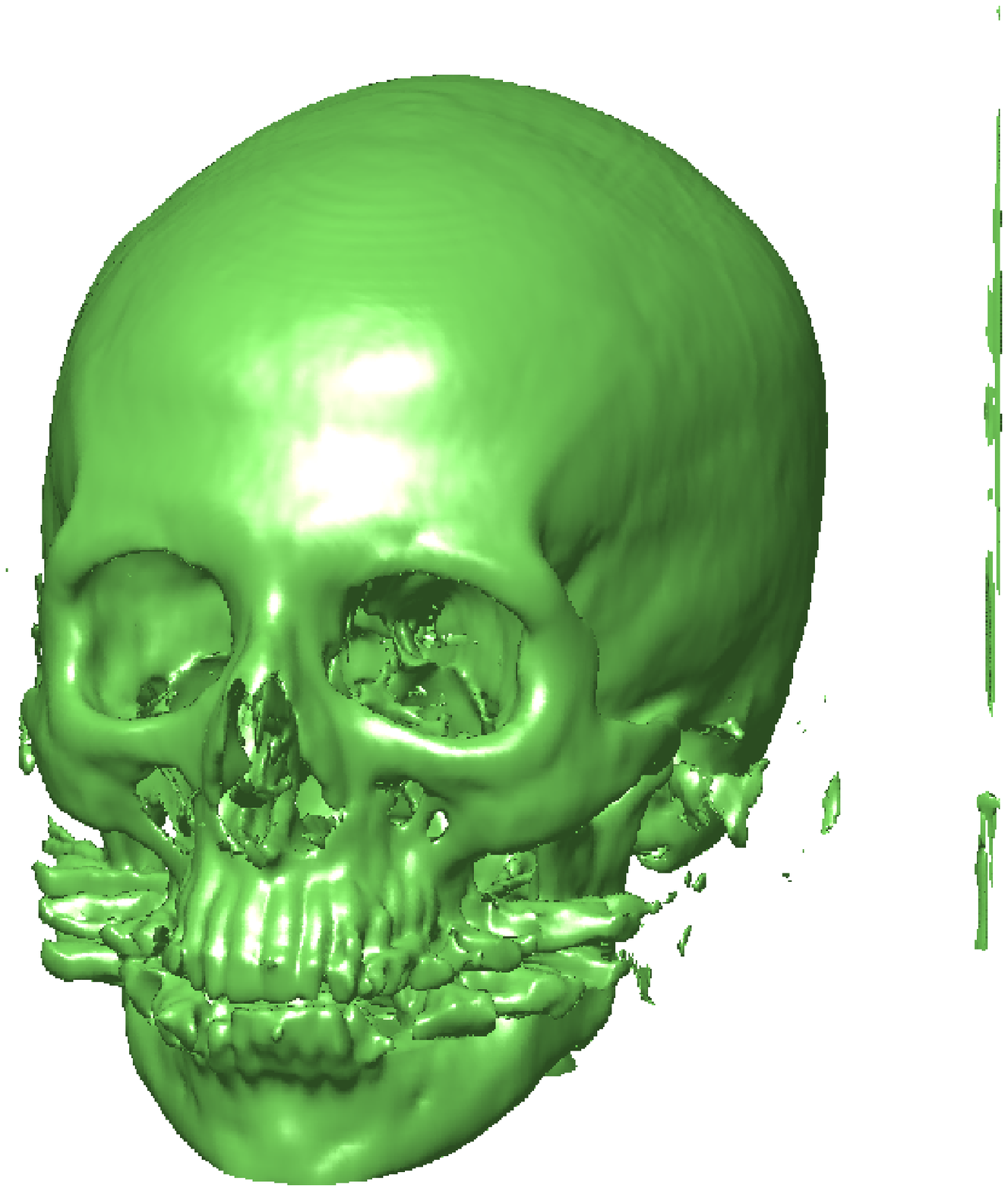}
\centerline{$(b)$}
\end{minipage}
\caption{Isosurfaces of the $C^2$ trivariate quartic spline approximating the {\sl CT Head data set} with isovalues: $(a)$ $\rho = 60$, $(b)$ $\rho = 90$}
	\label{rw}
\end{figure}

\begin{figure}[ht]
\centering\includegraphics[width=9cm]{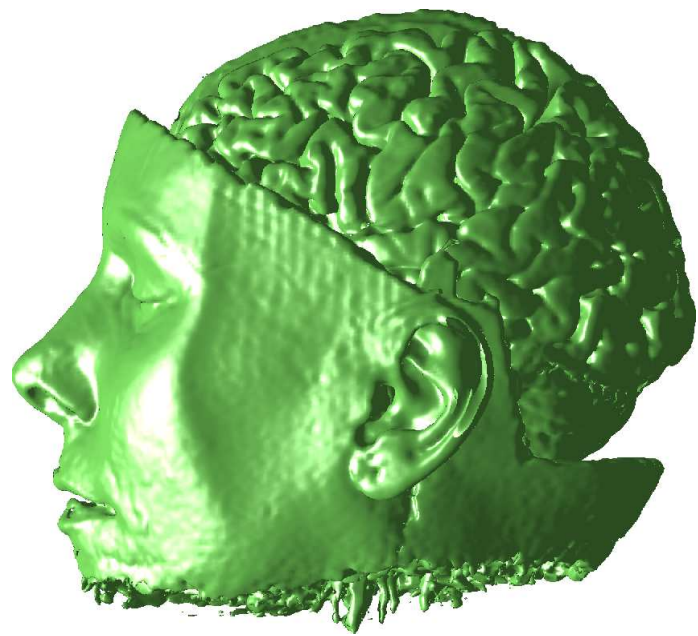}
\caption{Isosurface of the $C^2$ trivariate quartic spline approximating the {\sl MR brain data set} with isovalue  $\rho = 40$}
	\label{br}
\end{figure}

\clearpage

\section{Conclusions}
Given a 3D bounded domain, in this paper we have presented new QIs based on trivariate $C^2$ quartic box splines on type-6 tetrahedral partitions and with approximation order four. They are of near-best type, with coefficient functionals obtained by minimizing an upper bound for their infinity norm. Such QIs can be used for the reconstruction of gridded volume data and their higher smoothness is useful, for example, when functions have to be reconstructed with $C^2$ smoothness. Furthermore, trivariate $C^2$ splines can be used for constructing curvature continuous surfaces by tracing their zero sets.

Moreover, we have given norm and error bounds. Finally, some numerical tests, illustrating the approximation properties of the proposed quasi-interpolants, comparisons with other known spline methods and real world applications have been presented. 

\end{document}